\documentclass[a4paper,12pt]{article} 
\usepackage[affil-it]{authblk}
\usepackage{amsmath,amssymb,dsfont,mathtools}

\usepackage{graphicx}

\usepackage{graphicx}
\usepackage{enumerate}
\usepackage[noend]{algorithmic}
\usepackage[ruled]{algorithm}
\usepackage[usenames,dvipsnames]{xcolor}          

\usepackage{framed}

\usepackage{booktabs,dcolumn}
\definecolor{shadecolor}{rgb}{0.8,0.8,0.8}

\usepackage{geometry}
\usepackage{subfigure}

\newtheorem{theorem}{Theorem}[section]

\newtheorem{problem}{Problem}[section]
\newtheorem{proof}{Proof}[section]
\newtheorem{definition}{Definition}
\newtheorem{remark}{Remark}

\newcommand{\wth}{\ensuremath{\widetilde w}_h}
\newcommand{\zth}{\ensuremath{\widetilde z}_h}
\newcommand{\xth}{\ensuremath{\widetilde x}_h}

\title{An adaptive Newton algorithm for optimal control problems
with application to optimal electrode design}

\author{Thomas Carraro$^{1}$\thanks{thomas.carraro@iwr.uni-heidelberg.de}, Simon D\"orsam$^{1}$, Stefan Frei$^{1}$ and Daniel Schwarz$^{2}$}
\affil{\small$^1$Institute for Applied Mathematics, Heidelberg University}
\affil{$^1$Interdisciplinary Center for Scientific Computing (IWR), Heidelberg University}
\affil{$^2$Behavioural Neurophysiology, Max Planck Institute for Medical Research}
\affil{$^2$Department of Neuroradiology, Heidelberg University Hospital}
\affil{$^2$Department of Anatomy and Cell Biology, Faculty of Medicine, Heidelberg University}
\affil{69120 Heidelberg, Germany}

\begin{document}
\maketitle

\begin{abstract}
In this work we present an adaptive Newton-type method to solve nonlinear constrained optimization problems in which the constraint 
is a system of partial differential equations discretized by the finite element method. 
The adaptive strategy is based on a goal-oriented a posteriori error estimation for the discretization and for the iteration error. 
The iteration error stems from an inexact solution of the nonlinear system of first order optimality conditions by the Newton-type method. 
This strategy allows to balance the two errors and to derive effective stopping criteria for the Newton-iterations. 
The algorithm proceeds with the search of the optimal point on coarse grids which are refined only if the discretization 
error becomes dominant. Using computable error indicators the mesh is refined locally leading to a highly efficient solution process.
The performance of the algorithm is shown with several examples and in particular with an application in the neurosciences: 
the optimal electrode design for the study of neuronal networks.
\end{abstract}

\section{Introduction}
In this work we consider the optimal design of a glass micro-electrode for the use of reversible \textit{in vivo} electroporation in neural tissue.
Electroporation describes the increase in permeability of the cell membrane by the application of an external electric 
field beyond a certain threshold~\cite{Tsong1991,Weberetal2010}. While this technique has been known at least since the 1960’s~\cite{HamiltonSale1967},
it has become a standard tool in the neurosciences in more recent years to 
load single cells and small ensembles of neurons with a range of dyes and molecules, for example for the visualization of
neural networks~\cite{HaasSin2001,Nagayamaetal2007,Nevianetal2007}, see Figure \ref{neuronal network} on the left.

In order to make the plasma membrane permeable for a specific dye, the local voltage has to exceed a certain threshold. On the other hand the applied 
stimulus can not be increased infinitely, as high peaks of current would cause collateral damage~\cite{Nevianetal2007,GabrielTeissie1995}.
A way to reduce such unwanted side-effects is to modify the shape of the micro-electrodes, in order to obtain a more uniform
distribution of the electric field. While standard 
electrodes have a single hole at the tip, adding more holes on the side of the pipette seems a promising approach. 
Recent work has shown that nanoengineering 
techniques are indeed available to shape glass micro-electrodes in the tip region using focused ion beam assisted milling~\cite{Langford2006,Schwarz:2016}, 
see Figure \ref{neuronal network} on the right.
It has been shown that the part of the neuronal network, that can be visualized with these modified pipettes, 
is considerably enlarged in comparison to the standard design \cite{Schwarz:2016}, see Figure \ref{pip} for a numerical demonstration.

The objective of this work is to design an optimal electrode in terms of position and size of holes in the micro-pipette by using methods of numerical optimization.
The scientific contribution of this work is twofold: (i) on one side we present a mathematical formulation of the optimal design of a micro-pipette; 
(ii) on the other side we present an adaptive Newton method for the solution of the corresponding optimization problem.

The model used to describe the electric field is a partial differential equation (PDE). Therefore, we deal with a PDE constrained optimization problem. 
In the context of PDE constrained optimization problems the two common solution methods are the reduced and the all-at-once 
approach~\cite{HinzePinnauUlbrichUlbrich:2009}. 
We will adopt the latter one in which the optimality conditions are expressed as \textit{Karush-Kuhn-Tucker} (KKT) system.
There is a large literature on this topic and we refer for example to the books \cite{Fursikov:2000,HinzePinnauUlbrichUlbrich:2009,Luenberger:1969} for a thorough introduction.
Regarding the specific application there are no systematic studies that use a model based approach to design the micro-pipette used in electroporation. Therefore, 
the results shown here are of scientific interest even if they are obtained in a simplified setting with a two-dimensional problem. 
The extension to three-dimensional problems with a more complex model is possible within the same adaptive algorithm.

\begin{figure}[t]
\centering
 \includegraphics[height=0.25\textheight]{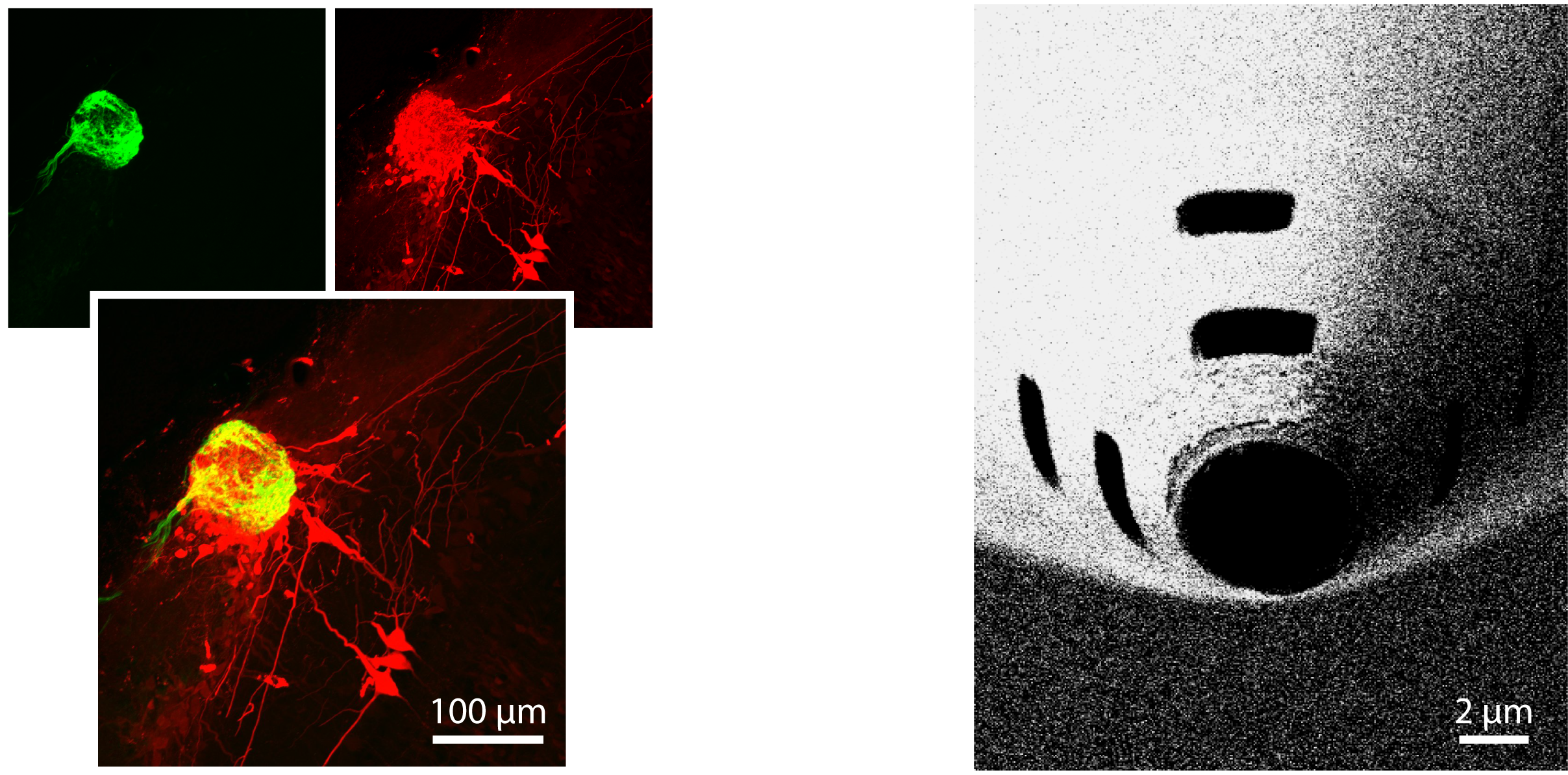}
 \caption{\textit{Left:} Example of a genetically-tagged olfactory glomerulus in the mouse as an example of a medium-sized neural 
  circuit in the brain (green, upper panel left). Upper panel right: Typical result after targeted electroporation of a 
  tetramethyl-rhodamine-dextran dye (red) revealing various types of 
  directly affiliated neurons and their processes in the surrounding region. Lower panel showing an overlay of the two fluorescent channels.
  \textit{Right:} Example of a modified glass micro-electrode after inserting several additional openings around the
  tip region by focused ion beam assisted milling.
  }
 \label{neuronal network}
\end{figure}

Mesh adaptivity is in many aspects well established in the context of finite element discretization of linear and nonlinear partial differential equations, 
see e.g.\ \cite{Babuska:2011,Verfuerth:2000}. Furthermore, goal oriented a posteriori error estimation has been successfully used in many applications, 
see the seminal works \cite{BeckerRannacher:2001,BangerthRannacher:2003} for an overview of the Dual Weighed Residual (DWR) technique and exemplarily
\cite{CarraroGoll:2017,Richter:2012,VexlerWollner:2008,BraackErn:2003} for some specific applications. 
A posteriori error estimation methods have been used to control the discretization error either in global norms, e.g.\ the $L^2$ or energy norm, 
or in specific functionals in the context of goal oriented techniques. 

To solve the nonlinear system arising from the discretization of the underlying problem typically a Newton-type method is used.  
If the Newton iteration is stopped after reaching a given tolerance, there is an iteration error that has to be taken into account in addition to the discretization error.
In particular, it is advantageous to control the iteration error and allow the Newton-iterates to stop before full convergence (i.e.$\,$to machine precision), because each Newton-iteration comes at the cost of the solution of a large linear system. 
The latter might be badly conditioned, especially in the context of multi-physics and optimization problems, 
leading to a large number of iterations 
of an iterative linear solver.
There are only few results on a posteriori error estimation that combine an estimation of
the discretization error and of the iteration error, resulting in algorithms that have stopping criteria based on balancing the two sources of error.

In the last few years increasing attention has been given
to adaptive strategies to solve nonlinear problems
including those arising from discretizations of partial differential equations.
Ziems and Ulbrich have presented in \cite{Ulbrich:2011} a class of inexact multilevel
trust-region sequential quadratic programming (SQP) methods for the solution of nonlinear 
PDE-constrained optimization problems, in which the discretization error in global norms
is controlled by local error estimators including control of the inexactness of the iterative solvers. 
Further works can be found outside the optimization context. A list of relevant publications is here given:

Bernardi and coauthors have shown an a posteriori analysis of iterative algorithms for nonlinear problems \cite{Bernardi:2015}, Rannacher and Vihharev have 
balanced the discretization error and the iteration error in a Newton-type solver \cite{RannacherVihharev:2013}; Ern and Vohral\'ik have developed an adaptive strategy for inexact 
Newton methods based on a posteriori error analysis \cite{Ern:2013} and Wihler and Amrein have presented an adaptive Newton-Galerkin method for semi-linear elliptic PDEs which combines an 
error estimation for the Newton step and an error estimation for the discretization with finite elements \cite{Wihler:2015}.

Since the goal of a simulation is the computation of a specific quantity of interest, for example in our case the optimal micro-pipette design (i.e.\ the position and dimension of the side holes), 
it is desirable to optimize the mesh refinement in a goal-oriented fashion. Furthermore, also the stopping criterion for the Newton iteration should be goal-oriented. This allows, for example in the context of optimization, 
to approximate the optimal point on coarse meshes and refine only once the discretization error becomes dominant.
In this way we reach the full balance of error sources with respect to the quantity of interest and the algorithm does the costly iterates (on fine meshes) only after the nonlinearities have 
been adequately solved on cheaper meshes. Consequently the computational costs are reduced by keeping the precision of the simulation at the desired level.
The new contribution of our work in this context is the derivation of a goal-oriented strategy for the adaptive control of a Newton-type algorithm to solve a nonlinear PDE-constrained optimization problem.

This work is organized as follows. In Section \ref{sec:Optimization problem} we formulate the general optimization problem; in Section \ref{sec.estimator} we present our adaptive strategy; 
in Section \ref{sec.application} we introduce the application in optimal electrode design; in Section \ref{sec.algo} we delineate the algorithms 
and in Section \ref{sec.num} we present some 
numerical results. Finally, in Section \ref{sec.conclusion}, an outlook to possible extensions of the presented method is given.

\begin{figure}[t]
\centering
 \includegraphics[width=5cm]{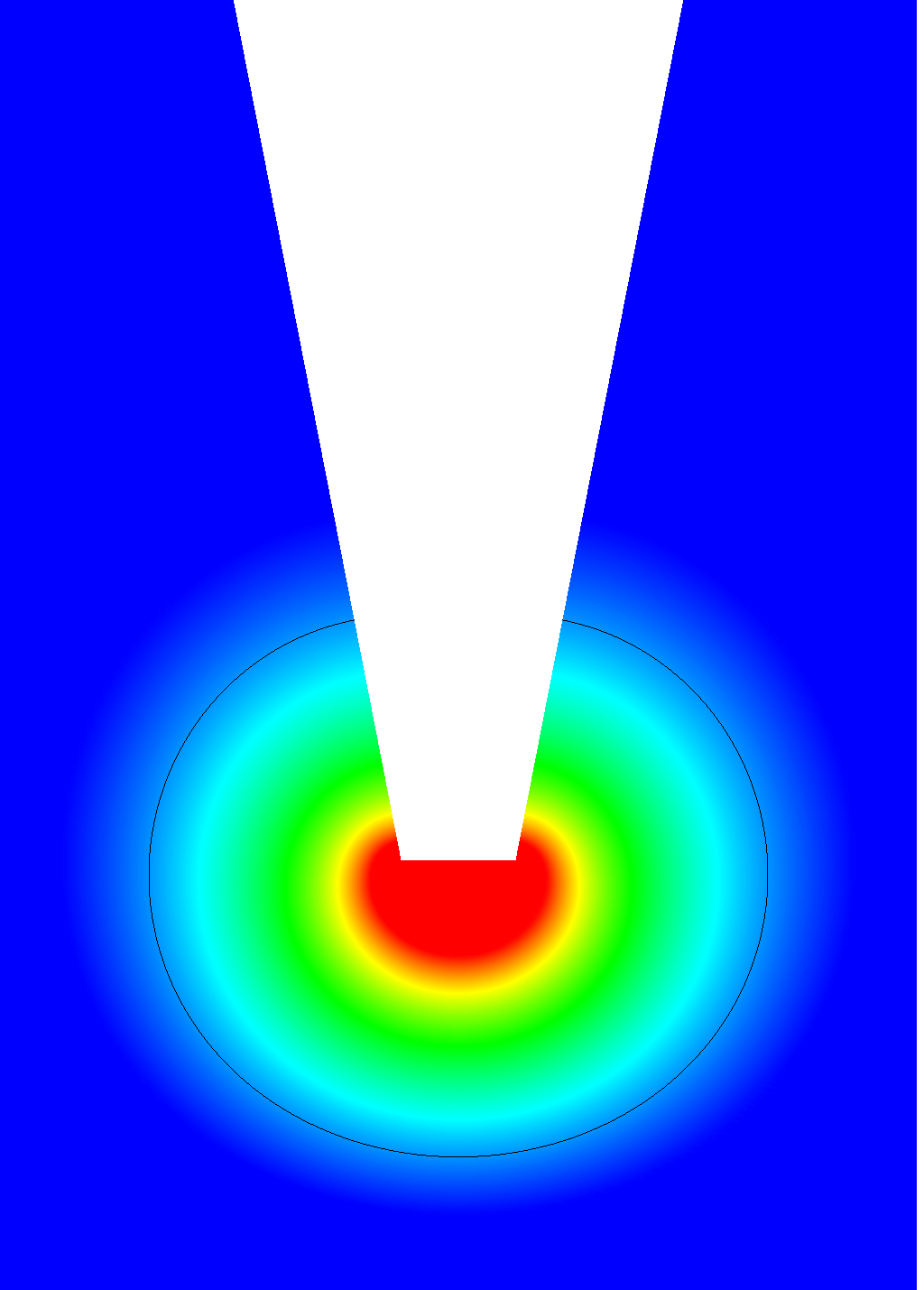}
 \includegraphics[width=5cm]{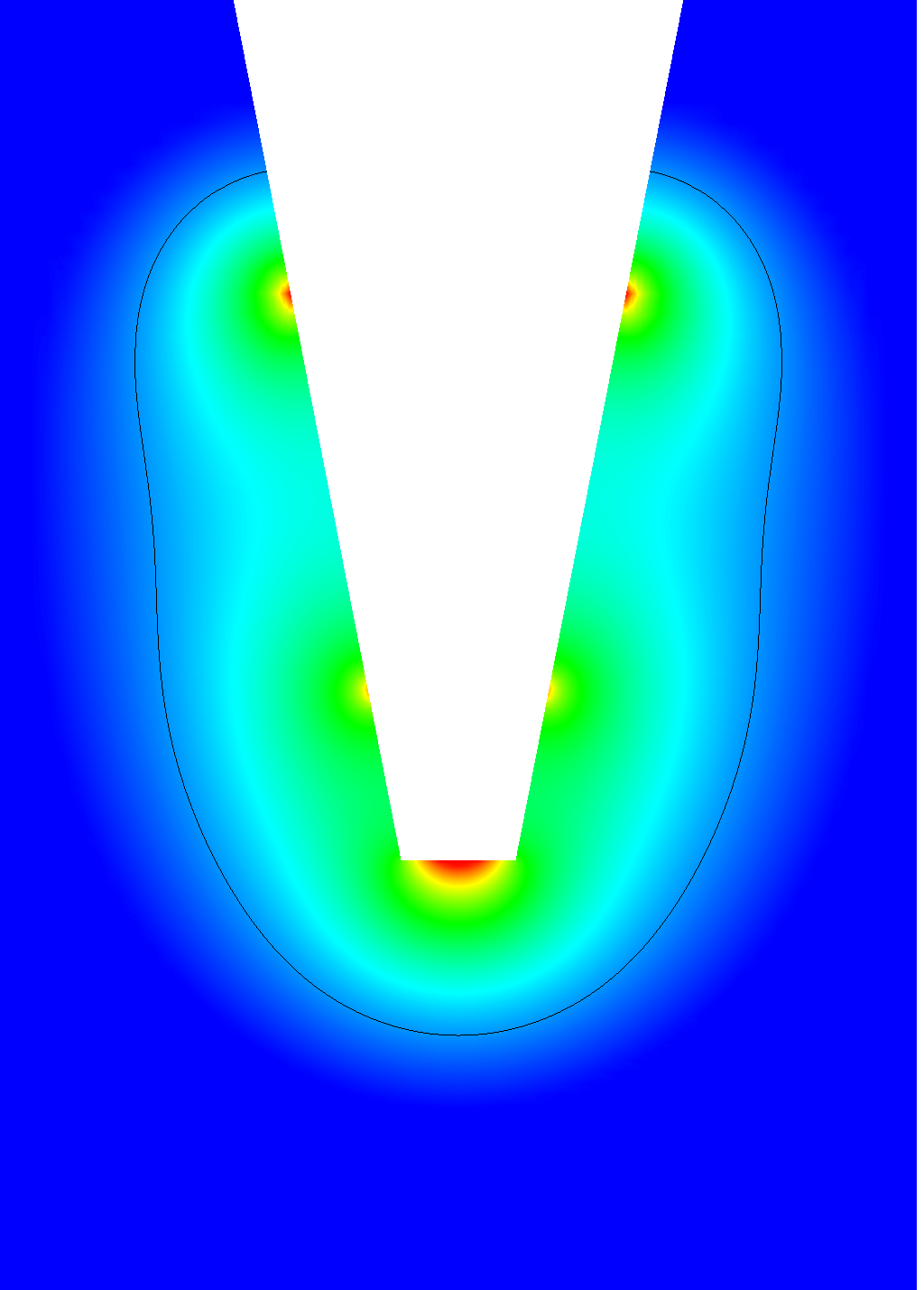}
 \caption{\label{pip}Two numerical results for the comparison of the activated region for a standard micro-pipette with one 
 hole only (left) and a modified micro-pipette with two additional set of holes (right). The black contour line
 illustrates the region, where a certain threshold is exceeded.}
 \end{figure}
\section{Optimization problem}
\label{sec:Optimization problem}
\noindent We consider the following optimization problem with parameters $q\in \mathbb{R}^s, s\in \mathbb{N}$
\begin{align}
\min_{q\in\mathbb{R}^s, u\in {\cal V}} J(u,q)& \label{opt}\\
s.t.\; A(u,q;\varphi) &=(f(q),\varphi) \quad \forall\varphi \in {\cal V}.\label{stateeq}
\end{align}
We assume that ${\cal V}$ is a reflexive Banach space. Let $A:{\cal V}\times \mathbb{R}^s\times {\cal V} \to \mathbb{R}$ be a semi-linear form and $f(q) \in {\cal V}^*$ for every $q\in\mathbb{R}^s$, where $\cal{V}^*$ denotes the dual space of $\cal V$.
Furthermore, we assume that $J$ and $A$ are twice (Fr\'echet) differentiable and that for each  $q\in \mathbb{R}^s$ the state equation (\ref{stateeq}) has a unique solution $u$. Let
us denote the (nonlinear) control-to-state map by $S:\mathbb{R}^s\to {\cal V}$.

Under these assumptions we can consider a reduced formulation of the optimization problem, with a reduced objective functional $j(q) := J(q,S(q)) : \mathbb R^s\rightarrow \mathbb R$. 
If the reduced objective functional is coercive the existence of local minimizers to (\ref{opt})-(\ref{stateeq}) follows by
standard arguments, see e.g.~\cite{Fursikov:2000,HinzePinnauUlbrichUlbrich:2009}.
The coercivity assumption is needed in case of unconstrained optimization problems to assure boundedness of the minimizing sequence. Therefore, for the practical solution of the problem, 
we consider a Tikhonov regularization term in the objective functional.
If in addition the functional is convex, the optimization problem has a unique solution.
Since in this work we allow nonlinearities in the model, we cannot assume convexity of the reduced functional. Therefore, the theoretical results assure only the existence of local minimizers.

\noindent To derive the optimality conditions, we introduce the Lagrange functional
\begin{align} 
  L \colon \mathcal{V} \times \mathbb{R}^s \times \mathcal{V} \to \mathbb{R}, 
  \quad L(u, q, \lambda) = J(u,q) + A(u,q; \lambda).
\end{align}
The first-order necessary optimality conditions are given by the KKT system
\begin{alignat}{2}
\label{KKT system}
\nonumber L'_{u}(u,q,\lambda)(\delta u)&=J'_u(u,q)(\delta u) + A'_u(u,q;\lambda)(\delta u)  =0 &&\quad\forall \delta u\in \mathcal{V},\\
L'_{q}(u,q,\lambda)(\delta q)&=J'_q(u,q)(\delta q) + A'_q(u,q;\lambda)(\delta q)  =0 &&\quad\forall \delta q  \in  \mathbb{R}^s,\\
\nonumber L'_{\lambda}(u,q,\lambda)(\delta \lambda)&= A(u,q;\delta \lambda)  =0 &&\quad\forall \delta \lambda\in \mathcal{V}.
\end{alignat} 

The first equation corresponds to the dual equation for the adjoint variable $\lambda$, the second equation
is called the control equation and the third equation
is the state equation for the primal variable $u$.

\subsection{Model problem}
To simplify the notation in the introduction of the error estimator in the next section, we consider a model problem of the form
\begin{align*}
\min_{q\in\mathbb{R}^s, u\in {\cal V}} J(u,q)&=\frac{1}{2} \int_{\Omega^s} \left(u-\hat{u}\right)^2 dx + \frac{\alpha}{2}|q|^2\\
s.t.\; \; \sigma (\nabla u,\nabla \varphi)_\Omega &=(f(q),\varphi)_\Omega \quad \forall\varphi \in {\cal V}
\end{align*}
where ${\cal V} := H^1_0(\Omega)$, $\Omega\subset \mathbb{R}^2$, $\alpha$ and $\sigma$ are positive real numbers, $(\cdot,\cdot)_\Omega$ 
denotes the $L^2$ scalar product and $\Omega^s \subset \Omega$.
The corresponding KKT system reads
\begin{problem}[KKT system of the model problem]
\label{KKT problem}
Find $w:=(u,q,\lambda) \in \mathcal{V}\times\mathbb{R}^s\times\mathcal{V}$ such that
\begin{alignat*}{2}
L'_{u}(w)(\delta u)&=(\delta u,u-\hat{u})_{\Omega^s} + \sigma ( \nabla \delta u, \nabla \lambda)_{\Omega}
=0  &&\quad \forall \delta u \in \mathcal{V}
,\\
L'_{q}(w)(\delta q)&=\alpha(\delta q,q)-(f'(q)(\delta q), \lambda)_{\Omega}
=0  &&\quad\forall \delta q \in  \mathbb{R}^s
,\\
L'_{\lambda}(w)(\delta \lambda)&= \sigma ( \nabla u, \nabla \delta \lambda)_{\Omega}
-(f(q),\delta \lambda)_{\Omega} =0 &&\quad\forall \delta \lambda \in \mathcal{V}
.
\end{alignat*}
\end{problem}
By introducing the semi-linear form
\begin{align}
\label{form A}
\begin{split}
\mathcal{A}(w;\delta w):=&(\delta u,u-\hat{u})_{\Omega^s} + \sigma ( \nabla \delta u, \nabla \lambda)_{\Omega}\\ 
&+ \alpha(\delta q,q)-(f'(q)(\delta q), \lambda)_{\Omega}\\
&+ \sigma ( \nabla u, \nabla \delta \lambda)_{\Omega}
-(f(q),\delta \lambda)_{\Omega}
\end{split}
\end{align}
we can write the KKT system in compact form as
\begin{align}
\label{compact KKT}
\mathcal{A}(w;\delta w) = 0 \quad \forall \delta w \text{ in } \mathcal{V}\times\mathbb{R}^s\times\mathcal{V}.
\end{align}

The derivation of a corresponding adaptive Newton method 
for other functionals $J$ and semi-linear forms $A$ fulfilling the assumptions made above is straight-forward
given that the KKT system is solvable with a Newton-type solver.
The modification of the optimization problem to the specific application presented in this paper
will be made later in Section \ref{sec.application}.

\subsection{Discretization}
We choose conforming finite element spaces ${\cal V}_h \subset {\cal V}$
for the state variable $u_h$ and the dual variable $\lambda_h$. The control space $\mathbb{R}^s$ is already 
finite dimensional, therefore we do not need a discretization of the control variable.
The discrete optimality system reads
\begin{problem}[Discrete KKT system of the model problem]
\label{discrete compact KKT system}
Find $u_h \in \mathcal{V}_h$, $q_h \in \mathbb{R}^s$ and $\lambda_h \in \mathcal{V}_h$, such that
\begin{align}
\label{compact discrete KKT}
\mathcal{A}(w_h;\delta w) = 0 \quad \forall \delta w \text{ in } \mathcal{V}_h\times\mathbb{R}^s\times\mathcal{V}_h.
\end{align}
\end{problem}
An essential problem in solving a discretized PDE system is the choice of the computational mesh on which depends the discretization error, i.e.\ the error due to the finite dimensional approximation given by the finite elements.

\section{Adaptive strategy}
\label{sec.estimator}
In the case of optimization problems it is of interest to control the accuracy of the solution of the first-order 
optimality conditions. The accuracy depends on the discretization error and it ``measures'' the quality of the approximation 
of the optimal point, i.e.\ of the optimal control and optimal state.
In the context of PDE constrained optimization problems, the 
two typical methods to solve the problem are the reduced approach and the all-at-once approach.
Here we use the \textit{all-at-once} approach, in which the optimality conditions are expressed
in terms of the gradient of the Lagrangian functional $L$ defined in the previous section. In particular, in absence of control and/or state constraints the optimality conditions are given by
\[
\nabla L(w)(\delta w) \overset{!}{=} 0 \quad \forall \delta w \text{ in } \mathcal{V}\times \mathbb{R}^s\times\mathcal{V},
\]
and the discrete counterpart is
\[
\nabla L(w_h)(\delta w) \overset{!}{=} 0 \quad \forall \delta w \text{ in } \mathcal{V}_h\times\mathbb{R}^s\times\mathcal{V}_h .
\]
Since the discrete approximation $(u_h, q_h, \lambda_h)$ is accurate only up to a certain tolerance 
that depends on the actual mesh refinement, it makes sense for efficiency reasons to
solve the optimality system only up to a certain accuracy as well.

The idea of our adaptive inexact Newton-type method is to balance the accuracy of the first order optimality conditions, i.e.\ of the KKT system, 
with the accuracy of its discrete approximation with respect to a goal functional, rather than with respect to some (global) norms of the solution or of the residuals. This is possible exploiting the flexibility 
of the DWR which allows to control the error with respect to an arbitrary functional.

In Section \ref{intro DWR}, we briefly introduce the DWR method and in Section \ref{DWR balance} we explain 
how to split the error into two contributions: one from the mesh discretization and the other from the inexact solution of the KKT system.

\subsection{Dual weighted residual (DWR) method}
\label{intro DWR}
We are interested in estimating the error $e(u,q,\lambda)$ measured in a quantity of interest: $$e(u,q,\lambda) := \mathcal{I}(u,q,\lambda) - \mathcal{I}(u_h,q_h,\lambda_h).$$ 
Following the seminal work of Becker and Rannacher \cite{BeckerRannacher:2001} we obtain the error identity by weighting the residual of the KKT system by an appropriate dual problem.
Let $w=(u,q,\lambda)$ be the solution of the KKT system \eqref{KKT system}. For the DWR error representation
we need the residual of the system, $\rho(w_h)(\cdot): \mathcal{V}\times\mathbb{R}^s\times\mathcal{V} \rightarrow \mathbb  R$, defined by
\begin{align}
\label{rho}
\rho(w_h)(\varphi) := 
\nonumber & \; (\varphi^u,u_h-\hat{u})_{\Omega^s} + \sigma ( \nabla \varphi^u, \nabla \lambda_h)_{\Omega}\\
& + \alpha(\varphi^q,q_h)-(f'(q_h)(\varphi^q), \lambda_h)_{\Omega}\\
\nonumber & + \sigma ( \nabla u_h, \nabla \varphi^\lambda)_{\Omega} -(f(q_h),\varphi^\lambda)_{\Omega}
\end{align}
with $\varphi=(\varphi^u,\varphi^q,\varphi^{\lambda})\in \mathcal{V}\times\mathbb{R}^s\times\mathcal{V}$.
Furthermore, we need the following adjoint problem to define the error estimator
\begin{problem}[Dual problem]
\label{dual problem}
Find $z:=(z^u,z^q,z^{\lambda}) \in \mathcal{V}\times\mathbb{R}^s\times\mathcal{V}$ such that
\begin{alignat*}{2}
 (z^u,\delta u)_{\Omega^s}+  \sigma ( \nabla z^{\lambda},  \nabla \delta u)_{\Omega} &=-\mathcal{I}'_u(u,q,\lambda)(\delta u) \qquad&& \forall \delta u \in \mathcal{V},\\ 
\alpha(z^q,\delta q)_{\Omega} -(z^{\lambda},f'(q)(\delta q))_{\Omega} -(\lambda, f^{\prime\prime}(q)(\delta q))_\Omega&=-\mathcal{I}'_q(u,q,\lambda)(\delta q)\qquad&& \forall \delta q \in \mathbb{R}^s,\\
 \sigma ( \nabla z^u, \nabla \delta \lambda)_{\Omega}- (\delta \lambda, f'(q)(z^q))_{\Omega}&=-\mathcal{I}'_{\lambda}(u,q,\lambda)(\delta \lambda)\qquad&& \forall \delta \lambda \in \mathcal{V}.
\end{alignat*}
By setting $\delta w = (\delta u,\delta q, \delta \lambda)$, the dual system reads 
 \begin{align}
 \label{dual KKT}
  \mathcal{A}^* (z,w)(\delta w) = -\mathcal{I}'_w(w)(\delta w) \quad \forall \delta w \in \mathcal{V} \times \mathbb{R}^s \times \mathcal{V} 
 \end{align}
with the adjoint bilinear form $\mathcal{A}^*(\cdot,\cdot)(\cdot): \bigl(\mathcal{V}\times\mathbb{R}^s\times\mathcal{V}\bigr)^3\rightarrow \mathbb R$ 
defined as
\begin{align*}
  \mathcal{A}^*(z,w)(\delta w) &:= (z^u,\delta u)_{\Omega^s}+  \sigma ( \nabla z^{\lambda},  \nabla \delta u)_{\Omega} 
    +\alpha(z^q,\delta q)_{\Omega} -(f'(q)(\delta q),z^{\lambda})_{\Omega}\\
    &\quad-(\lambda, f^{\prime\prime}(q)(\delta q))_\Omega+\sigma ( \nabla z^u, \nabla \delta \lambda)_{\Omega}- (f'(q)(z^q), \delta \lambda)_{\Omega}.
\end{align*}
\end{problem}
Its discretized counterpart is
\begin{problem}[Discretized dual problem]
\label{discrete dual problem}
Find $z_h:=(z^u_h,z^q_h,z^{\lambda}_h) \in \mathcal{V}_h\times\mathbb{R}^s\times\mathcal{V}_h$ such that
\begin{align}
\label{compact discrete dual KKT}
\mathcal{A}^* (z_h,w_h)(\delta w) = -\mathcal{I}'_w(w_h)(\delta w) \quad \forall \delta w \in \mathcal{V}_h \times \mathbb{R}^s \times \mathcal{V}_h.
\end{align}
\end{problem}
Since the model problem is nonlinear in $q$ we need to define the following dual residual $\rho^*(w_h, z_h)(\cdot): \mathcal{V}\times\mathbb{R}^s\times\mathcal{V}\rightarrow \mathbb R$
to derive the error estimator
\begin{align}
\label{rho*}
\rho^*(w_h, z_h)(\psi) := &
\nonumber \bigl(z^u_h,\psi^u\bigr)_{\Omega^s}+  \sigma \bigl( \nabla z^{\lambda}_h,  \nabla \psi^u\bigr)_{\Omega} + \mathcal{I}'_u(w_h)(\psi^u) \\
& + \alpha\bigl(z^q_h,\psi^q\bigr)_{\Omega} -\bigl(z^{\lambda}_h,f'(q_h)(\psi^q)\bigr)_{\Omega} -\bigl(\lambda_h, f^{\prime\prime}(q_h)(\psi^q)\bigr)_\Omega+ \mathcal{I}'_q(w_h)(\psi^q) \\
\nonumber & + \sigma \bigl( \nabla z^u_h, \nabla \psi^\lambda\bigr)_{\Omega}- \bigl(f'(q_h)(z^q), \psi^\lambda\bigr)_{\Omega} + \mathcal{I}'_{\lambda}(w_h)(\psi^\lambda),
\end{align}
with $\psi = (\psi^u, \psi^q, \psi^\lambda) \in \mathcal{V}\times\mathbb{R}^s\times\mathcal{V}$.

With these definitions, following \cite[Proposition 6.2]{BangerthRannacher:2003}, we get the error estimator

\begin{theorem}[A posteriori error estimator]
\label{dwr thm}
Let $w$, $w_h$ be the solutions of Problem \ref{KKT problem} and \ref{discrete compact KKT system} and let $z$, $z_h$ be the solutions of the continuous dual problem \ref{dual problem} 
and its discretized version \ref{discrete dual problem}. It holds the error identity
\begin{align}
\label{error identity}
\mathcal{I}(w) - \mathcal{I}(w_h) = \frac{1}{2} \rho(w_h)(z - z_h) + \frac{1}{2} \rho^*(w_h, z_h)(w - w_h) + R
\end{align}
with the residual $\rho(w_h)(\cdot)$ and the adjoint residual $\rho^*(w_h, z_h)(\cdot)$ defined in \eqref{rho} and \eqref{rho*}. The remainder term is given by
\begin{align}
\label{remainder}
R = \frac{1}{2} \int_0^1\bigl\{ \mathcal{I}^{\prime\prime\prime}(w_h+se)(e,e,e) - \mathcal{A}^{\prime\prime\prime}(w_h+ se;z_h+se^*)(e,e,e) - 3 \mathcal{A}^{\prime\prime}(w_h+se;e^*)(e,e)\bigr\}s(s-1)\mathrm{d}s
\end{align}
where $\mathcal{A}$ is the semi-linear form \eqref{form A} and the primal and dual errors are $e:=w - w_h$ and $e^*:=z - z_h$.
\end{theorem}
\begin{proof}
The proof follows by application of Proposition 6.1 from \cite{BangerthRannacher:2003} with the following Lagrange functional
\begin{align*}
 \mathcal{L}(u, q, \lambda,  z^u, z^q, z^{\lambda}) = \mathcal{I}(u, q, \lambda) - L'_{u}(u, q, \lambda)(z^u)
 - L'_{\lambda}(u, q, \lambda)(z^{\lambda}) - L'_{q}(u, q, \lambda)(z^q).
\end{align*} 
We sketch it here for later purposes.
Introducing the notation $x=(w,z)$ and $x_h=(w_h, z_h)$ and reminding the definition of the semi-linear form $\mathcal{A}$, see expression \eqref{form A}, we can rewrite it as
\begin{align*}
\mathcal{L}(x, z) := \mathcal{I}(x) - \mathcal{A}(x; z)
\end{align*} 
Furthermore, it is
\begin{align*}
\mathcal{I}(w) - \mathcal{I}(w_h) = \mathcal{L}(x) + \mathcal{A}(w; z) - \mathcal{L}(x_h) - \mathcal{A}(w_h; z_h) 
= \mathcal{L}(x) - \mathcal{L}(x_h),
\end{align*}
where we have used the fact that $w$ and $w_h$ satisfy \eqref{compact KKT} and \eqref{compact discrete KKT} respectively.
Considering the relation
\begin{align*}
\mathcal{L}(x) - \mathcal{L}(x_h) = \int_0^1\mathcal{L}^\prime(x + s(x-x_h))(e) \mathrm{d} s,
\end{align*}
the error identity follows from the error representation of the trapezoidal rule
\begin{align*}
\int_0^1 f(s)\mathrm{d}s = \frac{1}{2}\big( f(0) + f(1) \big) + \frac{1}{2}\int_0^1 f^{\prime\prime}(s)s(s-1)\mathrm{d}s.
\end{align*}
In fact, since $\mathcal{L}^\prime(x)(e)=0$ it is
\begin{align*}
\mathcal{I}(w) - \mathcal{I}(w_h) = \mathcal{L}(x) - \mathcal{L}(x_h) = \frac{1}{2} \mathcal{L}^\prime(x_h)(x - x_h) + R,
\end{align*}
where $R$ is the remainder term of the trapezoidal rule.
From this relation, using the definitions \eqref{rho} and \eqref{rho*}, the identity \eqref{error identity} can be deduced observing that
\begin{align*}
\mathcal{L}^\prime(x_h)(\cdot) = \mathcal{I}^\prime(x_h)(\cdot) - \mathcal{A}^\prime(x_h; z_h)(\cdot) - \mathcal{A}(x_h;\cdot).
\end{align*}
\end{proof}

\subsection{Balancing of discretization and iteration error}
\label{DWR balance}
In this work, we consider an inexact Newton-type method to solve the nonlinear KKT system \eqref{discrete compact KKT system}. 
We introduce the notation $\wth$ to indicate the inexact solution of the KKT system, which is obtained when the stopping criterion
\begin{align}
\label{inexact FE projection}
\frac{|\mathcal{A}(\wth;\delta w)|}{\|\delta w\|} \leq TOL \quad \forall \delta w \text{ in } \mathcal{V}_h\times\mathbb{R}^s\times\mathcal{V}_h
\end{align}
is reached and the notation $\zth$ to indicate the ``perturbed'' dual solution 
obtained by solving exactly (up to machine precision) the ``perturbed dual equation''
\begin{align*}
\mathcal{A}^* (\zth,\wth)(\delta w) = \mathcal{I}'_w(\wth)(\delta w) \quad \forall \delta w \in \mathcal{V}_h \times \mathbb{R}^s \times \mathcal{V}_h.
\end{align*}
We use the term ``perturbed dual equation'' for the adjoint equation in which we set the inexact primal solution $\wth$ as coefficient.

Since $\wth$ and $\zth$ are approximations of $w_h$ and $z_h$, an additional term appears
in the error identity \eqref{error identity} that accounts for the inexact Galerkin projection \eqref{inexact FE projection}.

Following \cite[Proposition 3.1]{RannacherVihharev:2013} we have the error estimator
\begin{theorem}[Error estimator with inexact Galerkin projection]
\begin{align}
\label{error identity with inexact Galerkin projection}
\mathcal{I}(w) - \mathcal{I}(\wth) = \frac{1}{2} \rho(\wth)(z - \zth) + \frac{1}{2} \rho^*(\wth, \zth)(w - \wth) - \rho(\wth)(\zth) + R,
\end{align}
with the residuals of the primal problem \eqref{rho} and of the dual problem \eqref{rho*} and the remainder term as in Problem \ref{dwr thm}.
\end{theorem}
\begin{proof}
Introducing the notation $x=(w,z)$, $\xth=(\wth,\zth)$ and the Lagrangian as in Theorem \eqref{dwr thm}, the proof follows from \cite[Proposition 3.1]{RannacherVihharev:2013}.
Let us consider the Lagrangian
\begin{align*}
\mathcal{L}(x) := \mathcal{I}(w) - \mathcal{A}(w;z).
\end{align*} 
It follows that
\begin{align*}
\mathcal{I}(w) - \mathcal{I}(\wth) = \mathcal{L}(x) + \mathcal{A}(w; z) - \mathcal{L}(\xth) - \mathcal{A}(\wth; \zth) 
= \mathcal{L}(x) - \mathcal{L}(\xth) - \mathcal{A}(\wth; \zth),
\end{align*}
where we have used the fact that $w$ satisfies \eqref{compact KKT}, while equation \eqref{compact discrete KKT} is solved only approximately.
Considering the trapezoidal rule and its remainder term, we get analogously to Theorem \ref{dwr thm} the identity
\begin{align*}
\mathcal{I}(w) - \mathcal{I}(\wth) = \mathcal{L}(x) - \mathcal{L}(\xth)  - \mathcal{A}(\wth;\zth)= \frac{1}{2} \mathcal{L}^\prime(\xth)(x - \xth) - \mathcal{A}(\wth;\zth) + R,
\end{align*}
from which the error representation \eqref{error identity with inexact Galerkin projection} can be deduced.
\end{proof}

\begin{definition}[Splitting of the error estimator]
For ease of presentation of the results and to derive the adaptive Newton strategy we split the error estimator into two parts. 
These are identified with the discretization error $\eta_h$ and the error due to the inexact Newton solution of the discrete KKT system $\eta_{KKT}$:
\begin{align*}
\mathcal{I}(w) - \mathcal{I}(\wth) = \mathcal{I}(w) - \mathcal{I}(w_h) + \mathcal{I}(w_h) - \mathcal{I}(\wth) \approx \eta_h + \eta_{KKT} =: \eta.
\end{align*}
Furthermore, using the error identity \eqref{error identity with inexact Galerkin projection} we define
\begin{align}
\eta_h &:= \frac{1}{2} \rho(\wth)(z - \zth) + \frac{1}{2} \rho^*(\wth, \zth)(w - \wth),\label{etahgen}\\
\eta_{KKT} &:= - \rho(\wth)(\zth)\label{etaKKTgen}
\end{align}

\noindent To evaluate the quality of the error estimator, we use the {\bf effectivity index} 
$$I_{\rm eff} = \frac{\eta_h}{\bigl(\mathcal{I}(w)-\mathcal{I}(\wth)\bigr)}.$$ An index close to one means that the estimator is reliable.
In the numerical examples in Section~\ref{sec.num}, we will observe that the indicator $\eta_h$ has a good effectivity index already at the 
beginning of the Newton iterations, when the solution approximation is inaccurate.
\end{definition}

We conclude this section by anticipating that our adaptive strategy defined in the algorithms in Section \ref{sec.algo} exploits the error 
splitting and attempts to balance the two error contributions, i.e.\ to reach the balance $\eta_h \approx \eta_{KKT}$, during the 
Newton iterations. In this way the adaptive strategy attempts to reduce the goal functional of the optimization problem on coarse meshes and it proceeds with mesh refinement only if the discretization error is dominating.
\begin{remark}
In the residual \eqref{rho} the term related to the residuum of the control equation is always zero because the control is finite dimensional. We keep it in the error representation because the same term in the residual $\rho(\wth)(\zth)$ in \eqref{error identity with inexact Galerkin projection} is nonzero due to the inexact Galerkin projection. In fact, this term is essential to get a reliable error estimator. Also in the dual residual \eqref{rho*}
the control term is zero. We keep the full expression here for the sake of completeness in the case of an infinite dimensional control space. 
\end{remark}

\section{Application: Optimal electrode design}
\label{sec.application}

As already mentioned in the introduction, we apply our adaptive strategy to the
optimal design of a micro-pipette to be used in electroporation. The objective is to 
maximize the area around the micro-pipette, where the voltage exceeds a certain threshold $\overline{u}$, 
while on the other hand an upper bound for the voltage $u_{\infty}$ shall not be reached.

The micro-pipette is covered by an isolating material such that the current can only flow to the biological tissue through the
micro-pipette holes. While standard micro-pipettes have only one hole at the tip, holes can also be created on the sides
of the micro-pipette using nanoengineering techniques~\cite{Schwarz:2016}. 

As some of the parameters, as e.g.$\,$the conductivity of the medium $\sigma$, are only known in a very rough approximation, 
we cannot expect to obtain quantitative results at this stage. Therefore, it seems justified for a qualitative study to restrict the setting to
a two-dimensional simplification.
Furthermore, we restrict the possible design of the holes to a symmetric setting with the same size of the openings
on both sides. 
In this context we are interested in the position and size of the openings as design parameters.

The domain of interest is a region around the micro-pipette $\Omega^s\subset \mathbb{R}^2$, where neuronal cells might be activated. 
To reduce the influence of exterior boundary conditions, we use a larger box $\Omega\supset \Omega^s$
around the micro-pipette as simulation domain (Figure~\ref{simdom}), excluding the micro-pipette itself. 
If we choose the box $\Omega$ large enough, we can 
assume without loss of generality zero voltage at the outer boundary.

We assume that a fixed current $\overline{I}$ is applied at the top of the micro-pipette. 
Knowing the resistances of
the electrode, the approximation of the fluxes through the holes of the micro-pipette can be 
derived by physical laws (see the appendix) and are defined
as flux (Neumann) conditions at the boundary of $\Omega$. 
The fluxes are expressed as functions of the radii and positions of the holes (which are the control variables for the optimal design), 
therefore the control variables define the Neumann boundary conditions as a finite dimensional parametrization.

\subsection{Governing equations}

 \begin{figure}[t]
 \begin{center}
 \resizebox*{0.6\textwidth}{!}{\input{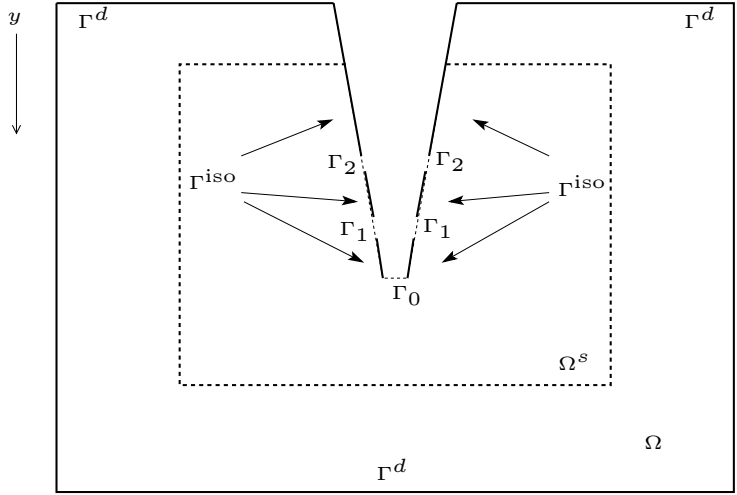}}
 \caption{\label{simdom}
  Simulation domain $\Omega$ with Dirichlet boundary $\Gamma^d$ and 
  Neumann boundary parts $\Gamma_{\rm pip}=\Gamma_0 \cup \Gamma_1 \cup \Gamma_2 \cup \Gamma^{\rm iso}$. The sub-domain $\Omega_s$
  is used to define the objective functional.}
 \end{center}
 \end{figure}

The governing equations can be derived as follows. In the absence of electric charge the Gauss law states
\begin{align}
 \sigma \, {\rm div} \, E = 0 \quad {\rm in }\, \Omega \label{ElectricField}
\end{align}
for the electric field $E: \Omega \to \mathbb{R}$ and the conductivity $\sigma \in \mathbb{R}_+$. With the electrostatic potential
$\varphi$ ($E= -\nabla \varphi$) (\ref{ElectricField}) reads
\begin{align}
 -\sigma \, \Delta \varphi = 0 \quad {\rm in}\, \Omega. \label{potential}
\end{align}
We assume zero Dirichlet conditions at the outer boundaries of the box denoted by $\Gamma^d$, 
see Figure~\ref{simdom}.
Then, we can rewrite (\ref{potential}) in terms of the voltage $u$
\begin{align*}
 -\sigma \Delta u = 0 \quad {\rm in}\, \Omega.
\end{align*}
The boundary condition at the micro-pipette is a flux condition
\begin{align*}
 \sigma \partial_n u = g \quad {\rm on }\, \Gamma_{\rm pip}.
\end{align*}
The flux $g$ is zero at the isolated parts $\Gamma^{\rm iso}$ of the micro-pipette. 
At the holes $\Gamma_0,\dots,\Gamma_n$, the flux is given by
\begin{align*}
 g= J_k,
\end{align*}
where $J_k$ denotes the current density at hole $\Gamma_k$. 
Let us assume that the number of holes is fixed.
Our design parameters $q \in \mathbb{R}^s$ will be the vertical position in terms of the midpoint $m_k$ of the 
holes and their sizes $s_k, k=0...s/2$.
The fluxes $J_k$ depend in a highly nonlinear way on $q$ (see the appendix). As the derivation of the analytical expressions $J_k(q)$ 
for more than two sets of holes are complex, we restrict our study to a maximum of two additional sets 
of holes besides the hole at the tip. 
A possible extension of this work would be to not rely on analytical expressions for the fluxes but extend the computational domain to the interior part of the micro-pipette and approximate the fluxes with a finite element discretization. 
Since the restriction to two sets of holes is not a limitation to show the effectiveness of our approach, we consider the analytical expressions derived in the appendix to reduce the computational effort.

\subsection{Objective functional}

Our aim is to maximize the region where the voltage exceeds a certain threshold $\overline{u}$. At the
same time, we have to ensure that the voltage does not exceed a critical value $u_{\infty} \gg \overline{u}$, with which the biological tissue
might be damaged. The corresponding objective functional would be
\begin{align*}
 J_{\chi}(u) = \int_{\Omega} \chi_M dx,
\end{align*}
where $\chi_M$ denotes the characteristic function of the set $M={\{x\in \Omega\, |\, u(x)\geq \overline{u}\}}$.
The main issue of the objective functional $J_{\chi}$ is its non-differentiability. For the gradient-based optimization algorithm
we present here, the functional is required to be at least differentiable.

Therefore, we consider another objective functional
\begin{align*}
 J_s(u) = \frac{1}{2} \int_{\Omega^s} \left( u - \hat{u} \right)^2 dx
\end{align*}
where we choose a constant function $\hat{u}>\overline{u}$ that is used as a tracking term to reach the desired threshold. 
Moreover, to avoid the influence of a far-away region, where we cannot expect that any cell can be activated, we consider for the functional definition the previously defined domain of interest, i.e.\ the sub-domain $\Omega^s \subset \Omega$ around the micro-pipette.

Using this functional, the voltage $u$ does not exceed the critical value $u_{\infty}$ in the numerical simulations conducted for this paper. 
In fact, we have found that using a value $\hat{u}$ slightly larger than the threshold $\overline{u}$ is a good choice to get
above to the threshold and to stay significantly below the critical value $u_{\infty}$.

The application poses additional restrictions for the design parameters $q$. Each hole has
to lie above the tip of the micro-pipette and below the upper end of the bounding box $y_{\rm tip}+s_k < m_k< y_{\rm up}-s_k$ and the holes should not
overlap $m_k+s_k < m_{k+1}-s_{k+1}$ (otherwise the formulas derived in the appendix for the fluxes at the boundary are not valid anymore).
In the numerical experiments conducted for this paper, however, these conditions were never violated. 
The addition of point-wise state constraints $u\leq u_{\infty}$ and/or control constraints $q\in Q_{\rm ad}\subset \mathbb R^s$, 
using an admissible set $Q_{\rm ad}$, can be done without significant changes in our approach using a penalty method and/or an
active set strategy \cite{ItoKunish:2008}.
Hence to simplify the exposition, we do not incorporate state and control constraints in this work.

Finally, we add a regularization term with parameter $\alpha>0$ to the objective functional as explained in Section \ref{sec:Optimization problem}.
The optimization problem reads in variational formulation
\begin{problem}[Optimization of the micro-pipette]
\label{Optimization of the micropipette}
Find the pair $u \in \mathcal{V}= H^1_0(\Omega; \Gamma^d)$ and $q \in \mathbb{R}^s$ that minimizes the goal functional $J$ under the PDE constrain:
\begin{align}
\begin{split}
\min_{u\in{\cal V},q\in\mathbb{R}^s}  J(u,q) &= \frac{1}{2} \int_{\Omega^s} \left(u-\hat{u}\right)^2 dx + \frac{\alpha}{2}|q|^2\\
\text{s.t. } \sigma ( \nabla u, \nabla \varphi)_{\Omega}&=  (g(q), \varphi)_{\Gamma_{\text{pip}}}  \quad \forall \varphi \in \mathcal{V}.
\label{ConcreteOpt}
\end{split}
\end{align}
\end{problem}

\subsection{Karush-Kuhn-Tucker system}
The Lagrange functional corresponding to problem \ref{Optimization of the micropipette} reads
\begin{align}
\label{Lagrangian pip}
L(u,\lambda,q) = J(u,q) + \sigma ( \nabla u, \nabla \lambda)_{\Omega}- (g(q),\lambda)_{\Gamma_{\text{pip}}}
\end{align}
with an adjoint variable $\lambda \in \mathcal{V} = H^1_0(\Omega;\Gamma^D)$. The first-order optimality conditions are given by:
\begin{problem}[First-order optimality conditions]
Find $u \in \mathcal{V}$, $q \in \mathbb{R}^s$, $\lambda\in \mathcal{V}$, such that
\begin{alignat}{2}
\label{KKTsystem_continuous}
\nonumber L'_{u}(u,q,\lambda)(\delta u)&=(\delta u,u-\hat{u})_{\Omega^s} + \sigma ( \nabla \delta u, \nabla \lambda)_{\Omega}=0  \quad&&\forall \delta u \in \mathcal{V}
,\\
L'_{q}(u,q,\lambda)(\delta q)&=\alpha(\delta q,q)-(g'(q)(\delta q), \lambda)_{\Gamma_{\text{pip}}}=0  \quad &&\forall \delta q \in  \mathbb{R}^s
,\\
\nonumber L'_{\lambda}(u,q,\lambda)(\delta \lambda)&= \sigma ( \nabla u, \nabla \delta \lambda)_{\Omega}- (g(q),\delta \lambda)_{\Gamma_{\text{pip}}} 
=0\quad &&\forall \delta \lambda \in \mathcal{V}
.
\end{alignat}
\end{problem}

\subsection{Discretization and approximation of the flux boundary conditions}
We use the space $\mathcal{V}_h$ of standard $Q_1$ finite elements on a quasi-uniform finite element mesh $\mathcal{T}_h$.
Altogether, the discrete optimality system reads:
\begin{problem}[Discrete first-order optimality conditions]
Find $u_h \in \mathcal{V}_h$, $q_h \in \mathbb{R}^s$, $\lambda_h \in \mathcal{V}_h$ such that
\begin{alignat}{2}
\label{KKTsystem}
\nonumber (\delta u_h,u_h-\hat{u})_{\Omega^s} + \sigma ( \nabla \delta u_h, \nabla \lambda_h)_{\Omega}
&=0&&\quad  \forall \,\delta u_h \in \mathcal{V}_h
,\\
\alpha(\delta q,q_h)-(g'(q_h)(\delta q), \lambda_h)_{\Gamma_{\text{pip}}}&=0&&\quad  \forall \, \delta q \in  \mathbb{R}^s
,\\
\nonumber \sigma ( \nabla u_h, \nabla \delta \lambda_h)_{\Omega}- (g(q_h),\delta \lambda_h)_{\Gamma_{\text{pip}}}
&=0 &&\quad \forall \,\delta \lambda_h \in \mathcal{V}_h.
\end{alignat}
\end{problem}

Denoting by $y$ the vertical position on the micro-pipette (see Figure \ref{simdom}) and by $y_{\rm tip}$ the position of the tip, 
it holds for the flux function $g$ that
\begin{align*}
 g(y,q) = \begin{cases}
         J_0(q), \quad &y=y_{\rm tip},\\
         J_k(q), \quad &| y - m_k| < s_k/2 \text{ for } k \in [1,...,s/2],\\
         0 \quad &\text{else}.
        \end{cases}
\end{align*}
Note that $g$ is discontinuous in both the vertical coordinate $y$
and the parameter vector $q$. This is a problem, since the derivative $g'(q)$ appears 
in the optimality system (\ref{KKTsystem}).
Furthermore, numerical methods like Newton-type methods for solving (\ref{KKTsystem}) require at least the first
derivative of the system which includes
$g''(q)$. To deal with this issue, we introduce a smooth approximation of $g$.
Let $\chi_{[-1,1]}$ be the characteristic function on the interval $[-1,1]$.
A smooth approximation to $\chi_{[-1,1]}$ is given by $\chi_{[-1,1]} \approx \exp(-x^{2\beta})$. 
Based on this approximation, we define a regularized flux function $\tilde{g}$ by 
\begin{align}
\begin{split}\label{expg}
 \tilde{g}(y,q) = \begin{cases}
         J_0(q), \quad &y=y_{\rm tip},\\
         \sum_{k=1}^s J_k(q) \exp\left(-\frac{(y - m_k)^{2\beta}}{4 s_k^2}\right) \quad &\text{else},
        \end{cases}
\end{split}
\end{align}
for some $\beta\in\mathbb{N}$, see Figure~\ref{fig.approx_char}.

\begin{figure}
  \centering
   \includegraphics[width=0.6\textwidth]{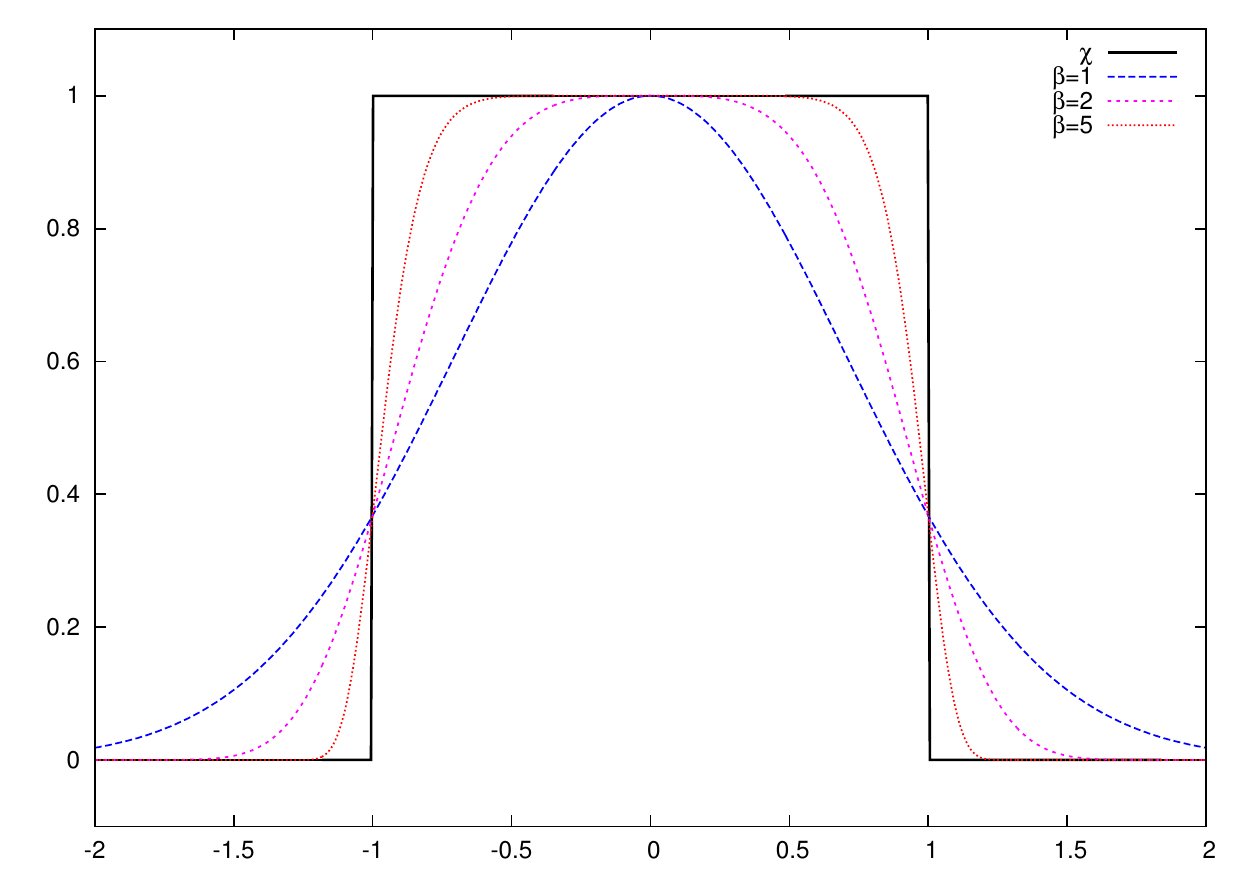}
 \caption{\label{fig.approx_char} Smooth approximation $\exp(-x^{2\beta})$ with $\beta=1,2,5$ of the characteristic function $\chi_{[-1,1]}$.}
\end{figure}

Furthermore, we use a summed quadrature formula with sufficient integration points to evaluate the boundary integrals such 
that the decay of $\tilde{g}$ at the boundary of the openings is appropriately approximated.

\subsection{Dual problem for the error estimation}
As explained in Section \ref{intro DWR} to estimate the discretization and iteration errors we need an approximation of the solution of an ad-hoc dual problem. In the specific case of the micro-pipette optimization the dual problem reads
\begin{problem}[Dual micro-pipette problem]
Find $z^u, z^\lambda \in H^1_0(\Omega;\Gamma^{\text{d}})$ and $z^q \in \mathbb R^s$ such that
\begin{alignat}{2}
\nonumber\bigl(z^u,\psi^u\bigr)_{\Omega^s}+  \sigma \bigl( \nabla z^{\lambda},  \nabla \psi^u\bigr)_{\Omega} &= - \mathcal{I}'_u(u,q,\lambda)(\psi^u) && \qquad \forall \psi^u \in \mathcal{V},\\
\alpha\bigl(z^q,\psi^q\bigr)_{\Omega} -\bigl(z^{\lambda},g'(q)(\psi^q)\bigr)_{\Gamma_{\text{pip}}} -\bigl(\lambda, g^{\prime\prime}(q)(\psi^q)\bigr)_{\Gamma_{\text{pip}}} &= - \mathcal{I}'_q(u,q,\lambda)(\psi^q) && \qquad \forall \psi^q \in \mathbb R^s,\\
\nonumber\sigma \bigl( \nabla z^u, \nabla \psi^\lambda\bigr)_{\Omega}- \bigl(g'(q)(z^q), \psi^\lambda\bigr)_{\Gamma_{\text{pip}}} &= - \mathcal{I}'_{\lambda}(u,q,\lambda)(\psi^\lambda) && \qquad \forall \psi^\lambda \in \mathcal{V}.
\end{alignat}
\end{problem}

This problem is discretized with finite elements to get the approximation $z_h=(z^u_h, z^\lambda_h, z^q_h)$. 
Furthermore, we observe that the system matrix of the dual problem is exactly the same matrix used in the Newton method 
to solve the primal problem, i.e.\ to solve the discrete KKT system \eqref{KKTsystem}. It is the Hessian of the Lagrange functional~\eqref{Lagrangian pip}.
It follows that the solution of the dual problem for the DWR method corresponds to one additional Newton step with
a different right-hand side, which will be explicitly shown later in the algorithmic section \ref{sec.algo}.

\subsection{A posteriori error estimators}
\label{sec.error estimator}
We conclude this section by specifying the concrete error estimators for the KKT system (\ref{KKTsystem}).
Following the derivation in Section~\ref{sec.estimator}, it holds that
\begin{align*}
\mathcal{I}(\tilde w_h) - \mathcal{I}(w)&\approx \eta_h+\eta_{KKT}
\end{align*}
with $\eta_h$ and $\eta_{KKT}$ specified in (\ref{etahgen}) and (\ref{etaKKTgen}).
An evaluation of the integrals over the cells leads in general to poor local error indicators due to the oscillatory nature of the residuals,
see \cite{CarstensenVerfuerth:1999}.
To avoid this behavior we integrate the residuals cell-wise by parts obtaining boundary terms (jump terms) to distribute the error 
on the inner cell-edges.
To simplify the notation we use the symbols without tilde implicitly considering that all quantities $w_h$ and $z_h$ are perturbed. Furthermore, we separate in (\ref{etahgen}) the contribution from the primal and the dual residuals, i.e.\ $\eta_h = \eta_h^p + \eta_h^d$ with
\begin{align}
\label{etah}
\nonumber\eta^p_h&= \sum_{K \in \mathcal{T}_h, K \subset \Omega_s} \big(\hat u - u_h, \Pi_h z_h^u \big)_K + \sum_{K \in \mathcal{T}_h} 
-\big( \sigma \Delta \lambda_h,\Pi_h z_h^u\big)_K + \big(r_h(\lambda_h),\Pi_h z_h^u\big)_{\partial K}\bigr\}\\
&\quad+ \sum_{K \in \mathcal{T}_h} \left\{ -\sigma \big(\Delta u_h,\Pi_h z_h^{\lambda}\big)_K + \big(r_h(u_h),\Pi_h z_h^{\lambda}\big)_{\partial K} 
- \big( g(q), \Pi_h z_h^{\lambda} \big)_{\partial K \cap \Gamma_{\text{pip}}}\right\}\\[1em]
\label{etahd}
\eta^d_h&=
\nonumber \sum_{K \in \mathcal{T}_h, K \subset \Omega_s} \big(z^u_h, \Pi_h u_h \big)_K + \sum_{K \in \mathcal{T}_h} 
-\big( \sigma \Delta z^\lambda_h,\Pi_h u_h\big)_K + \big(r_h(z_h^\lambda),\Pi_h u_h\big)_{\partial K}+ \mathcal{I}'_u(u,q,\lambda)(\Pi_h u_h) \bigr\}\\
&\quad+ \sum_{K \in \mathcal{T}_h} \left\{ -\sigma \big(\Delta z^u_h,\Pi_h {\lambda}_h\big)_K + \big(r_h(z_h^u),\Pi_h \lambda_h\big)_{\partial K} 
- \big( g'(q)(z^q), \Pi_h \lambda_h \big)_{\partial K \cap \Gamma_{\text{pip}}}+\mathcal{I}'_{\lambda}(u,q,\lambda)(\Pi_h \lambda_h ) \right\}
\end{align}
with the boundary residuals $r_h(\cdot)$ defined on $V_h$ by
\begin{alignat}{3}
 r_h(v_h)&=0 \text{  on } \Gamma^d,\quad r_h(v_h)&=\sigma \partial_n v_h \text{  on }  \Gamma_{\text{pip}},
 \quad r_h(v_h)&=\frac{\sigma}{2}[\partial_n v_h] \text{  on } \partial K \setminus (\Gamma^d \cup  \Gamma_{\text{pip}}).
\end{alignat}

In the error representation \eqref{error identity with inexact Galerkin projection} both the continuous primal and dual solutions $w$ and $z$ are used as weights. In fact, using the continuous weights and considering the remainder term, this expression is an identity. It becomes an estimation after substituting the unknown continuous solutions with computable quantities. As shown in several applications \cite{CarraroGoll:2017,Richter:2012,Richter:2010,Rannacher:2009,Carraro_RFDT:2007,BeckeBMRV:2007} one efficient method to produce computable quantities is the use of a patch-wise higher-order approximation of the terms $z_h$ and $w_h$.
In particular, we have used the following interpolation operator with $v_h \in V_h$
\begin{align*}
 \Pi_h v_h := (I_{2h}^{(2)} - \mathrm{id}) v_h,
\end{align*}
where $I_{2h}^{(2)}: V_h \to V_{2h}^{(2)}$ 
denotes the nodal interpolation into the space of quadratic polynomials on the patch mesh $\mathcal{T}_{2h}$ obtained by joining together four cells patchwise, as shown in Figure \ref{fig.hanging nodes}.

\begin{figure}[hbt]
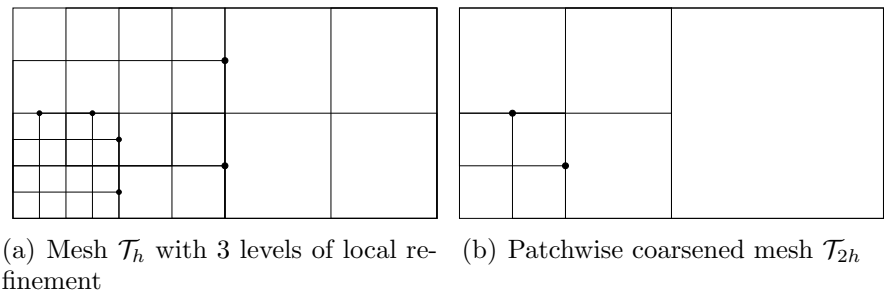

\centering
\subfigure[Mesh $\mathcal{T}_h$ with 3 levels of local refinement\label{subfig.fine grid}]{
\resizebox{0.35\textwidth}{!}{\input{PatchGitter}}}
\subfigure[Patchwise coarsened mesh $\mathcal{T}_{2h}$ \label{subfig.patch grid}]{
\resizebox{0.35\textwidth}{!}{\input{PatchGitter2}}}
\caption{{Example of a mesh with patch-structure and hanging nodes \subref{subfig.fine grid} and its patchwise coarsening \subref{subfig.patch grid}.}}\label{fig.hanging nodes}
\end{figure}

In the case of uniformly refined meshes and smooth primal and dual solutions, this weight-approximation has been
justified analytically in \cite{BangerthRannacher:2003}.

The estimator for the iteration error is defined by
\begin{align}
\begin{split}\label{etaKKT}
 \eta_{KKT} := -\mathcal{A}(w_h; z_h) &= -\sum_{K \in \mathcal{T}_h, K \subset \Omega_s} (\hat u -u_h,z_h^u)_K 
 + \sum_{K \in \mathcal{T}_h} \left\{ (\sigma \Delta \lambda_h,z_h^u)_K - (r_h^{\lambda},z_h^u)_{\partial K}\right\}\\
&\quad+ \sum_{K \in \mathcal{T}_h} \left\{ \sigma (\Delta u_h,z_h^{\lambda})_K - (r_h^u,z_h^{\lambda})_{\partial K}\right\}\\
&\quad- \alpha(q_h,z_h^q) + \sum_{K \in \mathcal{T}_h} \big(z_h^{\lambda},g'(q)(z_h^q)\big)_{\partial K \subset \Gamma_{\text{pip}}}.
\end{split}
\end{align}

\section{Algorithms}
\label{sec.algo}

In this section we introduce the algorithms that are compared in
Section~\ref{sec.num}. In addition to the fully adaptive algorithm, 
we will also specify a global refinement strategy and a purely mesh 
adaptive algorithm. The latter is the standard \textit{Dual Weighted Residual} method
applied to (\ref{KKTsystem_continuous}). 
The residual $\rho^k$ of the k-th iterate $w^k$ is defined by
\begin{align}
\begin{split}\label{residual}
\rho^k:=\rho(w^k)(\cdot):= 
\begin{pmatrix}
L'_{u}(w^k)(\cdot)
\\
L'_{q}(w^k)(\cdot)
\\
L'_{\lambda}(w^k)(\cdot)
\end{pmatrix}
\end{split}
\end{align}
The corresponding Hessian matrix is given by 
\begin{align}
\label{Hessian}
H:=H(w^k)(\cdot,\cdot)
&:=
\begin{pmatrix}
L^{\prime\prime}_{uu}      (w)(\cdot,\cdot) & L^{\prime\prime}_{uq}      (w)(\cdot,\cdot) & L^{\prime\prime}_{u\lambda}       (w)(\cdot,\cdot)\\
L^{\prime\prime}_{qu}      (w)(\cdot,\cdot) & L^{\prime\prime}_{qq}      (w)(\cdot,\cdot) & L^{\prime\prime}_{q\lambda}       (w)(\cdot,\cdot)\\
L^{\prime\prime}_{u\lambda}(w)(\cdot,\cdot) & L^{\prime\prime}_{q\lambda}(w)(\cdot,\cdot) & L^{\prime\prime}_{\lambda \lambda}(w)(\cdot,\cdot)\\
\end{pmatrix}
\\
&=
\begin{pmatrix}
 L^{\prime\prime}_{u u}(w)(\cdot,\cdot) & L^{\prime\prime}_{uq}(w)(\cdot,\cdot) & A'_{u}(w)(\cdot,\cdot)\\
 L^{\prime\prime}_{q u}(w)(\cdot,\cdot) & L^{\prime\prime}_{qq}(w)(\cdot,\cdot) & A'_{q}(w)(\cdot,\cdot)\\
 A'_{u}                (w)(\cdot,\cdot) & A'_{q}               (w)(\cdot,\cdot) & 0\\
\end{pmatrix}
\end{align}

Let a tolerance
$TOL_{KKT}$ for the Newton residual $\rho^k$ be given.
We formulate the algorithms for a damped Newton method with a damping parameter $\alpha_N\in(0,1]$.
The full Newton step is obtained for $\alpha_N=1$. 

In Algorithm \ref{algo:global} we present the global refinement strategy with a given number of refinement steps $n_{ref}$.
\begin{algorithm}[h!]
\begin{algorithmic}
\STATE Initialization:  Set $k=0$, choose $w^0=(u^0,\lambda^0,q^0)$ and a mesh $\Omega_h^0$.
\FOR{$l=0...n_{ref}$} 
\STATE Compute the residual $\rho^k=\rho(w^k)$ by (\ref{residual}). 
\WHILE{$\rho^k \geq TOL_{KKT}$} 
\STATE Compute the Hessian $H^k=H(w^k)$ by (\ref{Hessian}).
\STATE Solve $H^k \Delta w=-\rho^k$.
\STATE Update $w^{k+1}=w^k+\alpha_N \Delta w$. 
\STATE Compute the next residual $\rho^{k+1}=\rho(w^{k+1})$ by (\ref{residual}) and set $k=k+1$.
\ENDWHILE
\STATE {\bf end}
\STATE Refine the mesh $\Omega_h^k$ by global refinement, interpolate $w^{k+1}$ to the refined grid.
\ENDFOR
\STATE {\bf end}
\end{algorithmic}
\caption{Global refinement}
\label{algo:global}
\end{algorithm}
In this algorithm, neither the discretization nor the iteration error is estimated. 
Thus, we have no control over these errors and
the usual stopping criterion is based on a tolerance on a norm of the residual. 
This generic criterion does not allow to control the inexactness of the optimal solution that is needed 
to advance with the optimization on finer grids before machine precision is reached.

Next, we introduce the purely mesh adaptive Algorithm \ref{meshadaptive}. We introduce a further tolerance $TOL$ for the discretization error.
In addition to the notation introduced above, we
denote the right-hand side of the dual problem by
\begin{align}
\begin{split}\label{dualrhs}
\zeta^k:=\zeta(w^k)(\cdot):=
\begin{pmatrix}
{\cal I}'_{u}(w^k)(\cdot)
\\
{\cal I}'_{q}(w^k)(\cdot)
\\
{\cal I}'_{\lambda}(w^k)(\cdot)
\end{pmatrix}
.
\end{split}
\end{align}

\begin{algorithm}[h!]
\begin{algorithmic}
\STATE Initialization:  Set $k=0$, choose $w^0=(u^0,\lambda^0,q^0)$ and a mesh $\Omega_h^0$.
\WHILE{$\eta \geq TOL$} 
\STATE Compute the residual $\rho^k=\rho(w^k)$ by (\ref{residual}). 
\STATE Compute the Hessian $H^k=H(w^k)$ by (\ref{Hessian}).
\WHILE{$\rho^k \geq TOL_{KKT}$} 
\STATE Solve $H^k \Delta w=-\rho^k$.
\STATE Update $w^{k+1}=w^k+\alpha_N \Delta w$.
\STATE Compute the next residual $\rho^{k+1}=\rho(w^{k+1})$ by (\ref{residual}).
\STATE Compute the next Hessian $H^{k+1}=H(w^{k+1})$ by (\ref{Hessian}) and set $k=k+1$. 
\ENDWHILE
\STATE {\bf end}
\STATE Compute $\zeta^{k+1} = \zeta(w^{k+1})$ by (\ref{dualrhs}).
\STATE Solve the dual problem $H^{k+1} z^{k+1} = -\zeta^{k+1}$.
\STATE Evaluate the error estimator $\eta_h^{k+1}$ by (\ref{etah}) and refine the mesh adaptively.
\STATE Project the solution $w^{k+1}$ to the new grid $\Omega_h^{k+1}$.
\ENDWHILE
\STATE {\bf end}
\end{algorithmic}
\caption{Mesh adaptive}
\label{meshadaptive}
\end{algorithm}
Based on the error indicators several refinement strategies can be derived. We refer to \cite{BeckerRannacher:2001} for
an overview. In this work we use a refinement strategy based on a minimization of the expected error and
the computational effort required for the solution on the refined mesh, see \cite{Richter:Diss}.

Finally, we concretize the fully adaptive Algorithm \ref{fulladaptive} which is the novelty of this contribution.
\begin{algorithm}[h!]
\begin{algorithmic}
\STATE Initialization:  Set $k=0$, choose $w^0=(u^0,\lambda^0,q^0)$ and a mesh $\Omega_h^0$.
\WHILE{$\eta \geq TOL$} 
\STATE Compute the residual $\rho^k=\rho(w^k)$ by (\ref{residual}). 
\STATE Compute the Hessian $H^k=H(w^k)$ by (\ref{Hessian}).
\WHILE{$|\eta_{KKT}^k| \le c_b |\eta_{h}^k|$} 
\STATE Solve $H^k \Delta w=-\rho^k$.
\STATE Update $w^{k+1}=w^k+\alpha_N \Delta w$.
\STATE Compute the next residual $\rho^{k+1}=\rho(w^{k+1})$ by (\ref{residual}).
\STATE Compute the next Hessian $H^{k+1}=H(w^{k+1})$ by (\ref{Hessian}). 
\STATE Compute $\zeta^{k+1} = \zeta(w^{k+1})$ by (\ref{dualrhs}).
\STATE Solve the dual problem $H^{k+1} z^{k+1} = -\zeta^{k+1}$.
\STATE Evaluate the error estimator $\eta_h^{k+1}$ by (\ref{etah}) and set $k=k+1$. 
\ENDWHILE
\STATE {\bf end}
\STATE Refine the mesh adaptively and project the solution $w^{k+1}$ to the new grid $\Omega_h^{k+1}$.
\ENDWHILE
\STATE {\bf end}
\end{algorithmic}
\caption{Fully adaptive}
\label{fulladaptive}
\end{algorithm}
It remains to specify the constant $c_b$ in Algorithm~\ref{fulladaptive} to determine the balancing between
$\eta_h^k$ and $\eta_{KKT}^k$. A straightforward choice would be $c_b=1$, i.e.\ to stop the Newton iteration, as soon as the iteration
error is smaller than the discretization error. Nevertheless, we have decided in the numerical examples conducted for this paper
to use a smaller value of $c_b$. In this way, we obtain that the Newton method remains in 
the region of quadratic convergence on the next finer grid, once this convergence rate is achieved on the previous coarser one.
In the numerical examples below, we have used $c_b=0.1$. In general,
the optimal choice for $c_b$ depends on the specific application.

On the first sight, the two adaptive algorithms look very similar. The main difference is the stopping criterion of the second \textit{while}-loop which
depends on the balancing of the error estimators in the fully adaptive algorithm and on the Newton residual in the mesh-adaptive algorithm. 
The balancing criterion allows the Newton method to iterate on the actual grid as much as needed to reach the discretization error, therefore it saves at each level unneeded iterations.
However, it has to be noted that in the full adaptive algorithm a dual problem has to be solved within the second \textit{while}-loop in every iteration, 
while in the mesh-adaptive 
algorithm a dual problem is solved only once on each mesh level. We investigate in the next section how this additional 
computational effort compares to the computational savings due to the improved stopping criterion.
\section{Numerical results}
\label{sec.num}

In this section, we study different numerical examples to test the algorithms. First, we study a simple test problem on a slit domain
in Section~\ref{sec.num_slit}. The purpose of this test problem
is to test the error estimators $\eta_h$ and $\eta_{KKT}$. To this end, we compute
effectivity indices and investigate if the estimators are relatively independent of each other. 
Then we test the algorithms for optimal electrode design in Section~\ref{sec.num_pipette}.
We compare the fully adaptive algorithm to the two other refinement algorithms introduced
in Section~\ref{sec.algo} with respect to errors, degrees of freedom and computational times.
Finally, we compare the optimal results for 0, 1 and 2 sets of holes in addition to the opening at the tip.

All the numerical results presented in this section are obtained using the finite element library deal.II \cite{BangeHK:2007}.
To calculate the lengthy first and second derivatives of $g(q)$ with respect to the design parameter $q$
(see the appendix) we have used the automatic differentiation tool ADOL--C \cite{Walther:2012}. The runtimes 
shown were obtained on a desktop computer with a 2.66GHz Intel Core 2 Quad processor (Q9400). As linear solver
within the Newton algorithm, a direct solver 
is applied (UMFPACK~\cite{Davis:2004}).

\subsection{Test problem on a slit domain}
\label{sec.num_slit}

We start by studying a test problem on the slit domain
\[
\Omega= (0,1)^2 \setminus \{0.5,y) \,|\, 0<y<0.5\}, 
\]
see Figure~\ref{tab.slit_global} on the right.
We set homogeneous Dirichlet data on the outer boundary except for the upper part $\Gamma^{\rm top}$, 
where the Neumann condition $\partial_n u =q^2\pi\sin(\pi x)$ is imposed.
As objective functional, we set $J(u,q)=\frac{1}{2}\int_{\Omega} (u-u_0)^2 dx+ \frac{\alpha}{2}q^2$, 
where $u_0= \frac{1}{\sigma} \sin(\pi x)\sin(\pi y)$ and $\sigma=1.72$. The quantity of interest is $\mathcal{I}(u,q)=q^2$.

To test the estimator $\eta_h$, we first show results of the global refinement strategy in Figure~\ref{tab.slit_global} on the left. 
The optimal solution $u$ shows a singularity in the gradient at the top of the slit. Therefore, 
the error in the goal functional ${\cal I}(q)$ as well as the error estimator decrease linearly as expected (see e.g.$\,$\cite{DeuflhardWeiser}, Section 6.3.2).
As the Newton algorithm is solved up to 
a small tolerance $\rho^k < 10^{-10}$ in the Newton residual, it holds $\eta\approx\eta_h$. 
The effectivity indices $I_{\rm eff} = \eta / (I(q)-I(q_h))$ on the
finer grids are $\eta\approx0.32$. These values show that the error estimator estimates the order of magnitude of the error
correctly, the deviation of estimator and real error
being mainly due to the approximation of the weight $z-z_h$ by the interpolation $I_{2h}^2 z_h - z_h$.

\begin{figure}
\centering
\begin{minipage}{11cm}
 \begin{tabular}{c|cccccc} 
\#dofs & $\mathcal{I}(q)-\mathcal{I}(q_h)$  & $\eta_h$ & $\mathcal{I}_{eff}$ & \#steps & time[s]\\
\hline
54      & 1.5e-01 & 1.1e-01 &0.71 & 6 & 3\\
170     & 6.1e-02 & 2.9e-02 &0.47 & 4 & 5\\
594     & 2.7e-02 & 1.0e-02 &0.40 & 3 & 9\\
2210    & 1.2e-02 & 4.5e-03 &0.36 &3 & 18\\
8514    & 6.1e-03 & 2.0e-03 &0.34 &3 & 39 \\
33410   & 3.0e-03 & 9.8e-04 &0.32 & 2 & 84\\
132354  & 1.4e-03 & 4.8e-04 &0.32 & 2 & 266 \\
526850  & 7.4e-04 & 2.3e-04 &0.32 & 2 & 1157\\
\end{tabular}
\end{minipage}
\begin{minipage}{4cm}
 \resizebox*{0.75\textwidth}{!}{\input{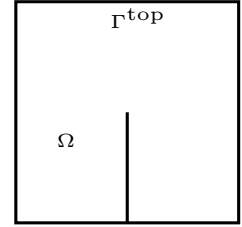}}
  \end{minipage}

\caption{Left: Error, error estimator, efficiency indices, number of Newton steps and computational times for the global refinement algorithm applied to the
test problem on the slit domain. Right: Sketch of the slit domain $\Omega$\label{tab.slit_global}.}
\end{figure}

In Figure~\ref{fig.plots_slit}, we compare the results obtained by global mesh refinement to the adaptive strategy introduced in Section~\ref{sec.algo}. 
We plot the error in ${\cal I}(q)$ against the degrees of freedom $N$ (left) and against the computational times
(right). The error for the two adaptive methods differ only marginally, hence we plot only the fully adaptive result 
in the left plot. In this example the error decreases approximately with $\mathcal{O}(N^{-1/2}) = \mathcal{O}(h)$ for
global refinement and with $\mathcal{O}(N^{-1})$ for the adaptive mesh refinement. 
While the global refinement strategy shows for example an error of $7.4\times10^{-4}$
for $526.850$ degrees of freedom, the adaptive strategy reaches an error of $5.6\cdot10^{-4}$ with only $7.722$ degrees of freedom.
The solution $u_{\rm opt}$ on an adaptively refined mesh is shown in Figure~\ref{fig.slit} on the right.

\begin{figure}[t]
 \centering
  \includegraphics[width=0.45\textwidth]{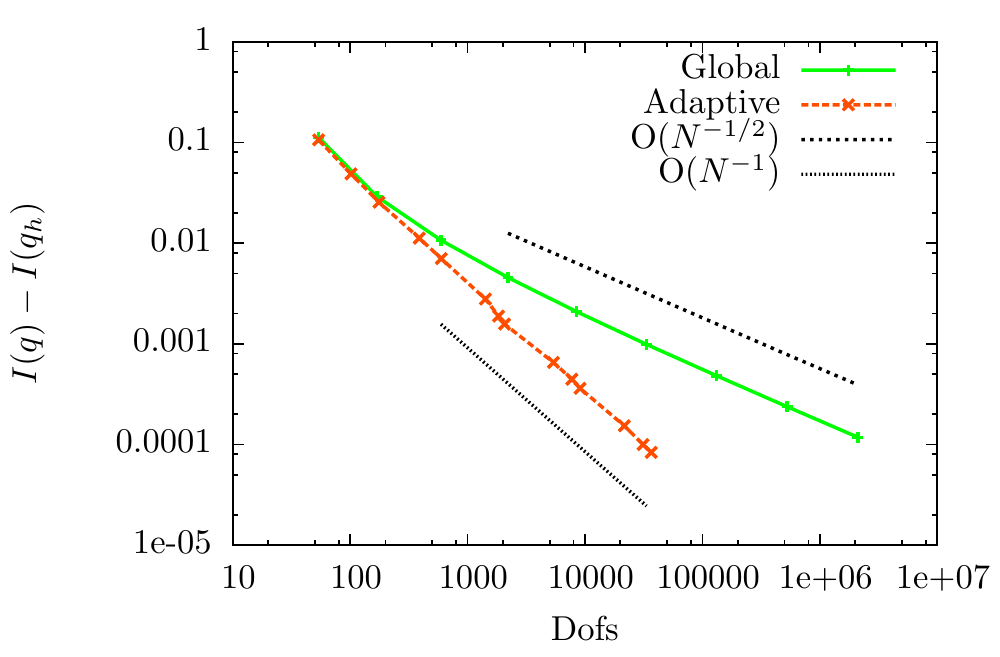}
   \includegraphics[width=0.45\textwidth]{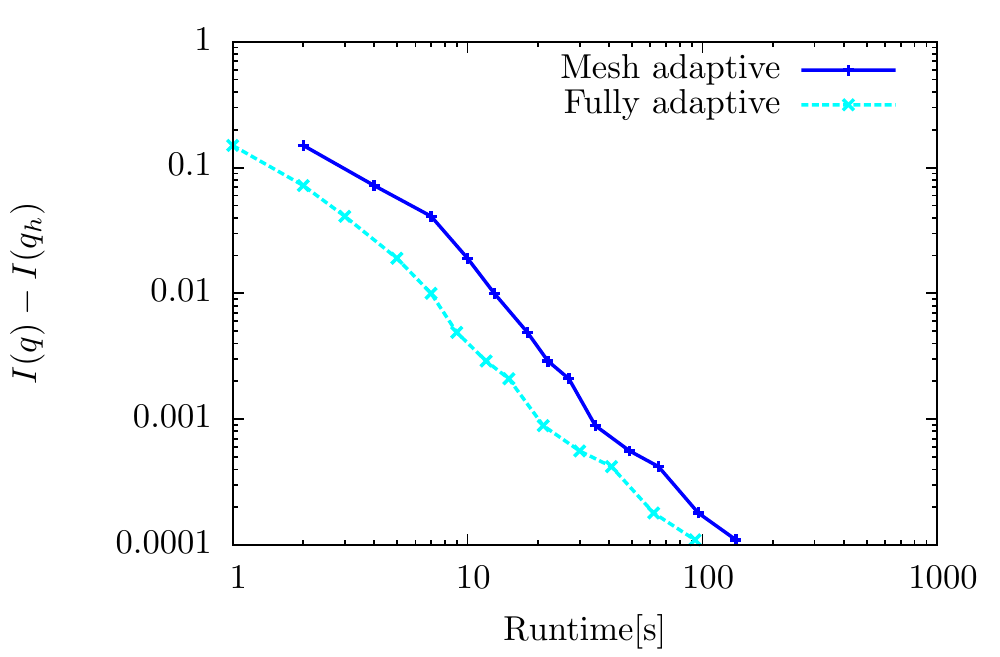}
  \caption{\textit{Left}: Error $I(q)-I(q_h)$ for global mesh refinement and 
  the fully adaptive refinement strategy for the test problem on the slit domain. For comparison we plot reference rates ${\mathcal O}(h^{-s})$
  for $s=1,2$. \textit{Right}: Computational times for the three strategies: global refinement, adaptive mesh refinement and adaptive
  Newton method \label{fig.plots_slit}
}
\end{figure}

With respect to computational times, the fully adaptive and the mesh-adaptive algorithm show asymptotically
the same convergence behavior, with a clear advantage of the fully adaptive algorithm. To study this difference
quantitatively, we show the number of Newton steps and the computational times in Figure~\ref{fig.slit} on the left.

\begin{figure}[t]
\renewcommand{\arraystretch}{1.2}
\centering
\begin{minipage}{9cm}
\small
\begin{tabular}{cc|cc|cc} 
 & & \multicolumn{2}{c|}{mesh adaptive} &\multicolumn{2}{c}{fully adaptive} \\
\hline
\#dofs &$\mathcal{I}(q)-\mathcal{I}(q_h)$ & \#steps & time[s] & \#steps & time[s]\\
\hline
54     & 1.5e-01 & 6 & 2 & 2 & 1\\
102    & 7.2e-02 & 4 & 4 & 1 & 2\\
176    & 4.1e-02 & 3 & 7 & 1 & 3\\
388    & 1.9e-02 & 3 & 10 & 1 & 5\\
598    & 1.0e-02 & 3 & 13 & 1 & 7\\
1422   & 4.9e-03 & 3 & 18 & 1 & 9\\
1840   & 2.9e-03 & 2 & 22 & 1 & 12\\
2068   & 2.1e-03 & 2 & 27 & 1 & 15\\
5398   & 8.9e-04 & 2 & 35 & 1 & 21\\
7722   & 5.6e-04 & 2 & 49 & 1 & 30\\
9080   & 4.2e-04 & 2 & 65 & 1 & 41\\
21592  & 1.8e-04 & 2 & 96 & 1 & 62\\
31192  & 1.1e-04 & 2 & 139 & 1 & 93\\
36680  & 9.0e-05 & 2 & 193 & 1 & 131\\
\end{tabular}
\end{minipage}
\begin{minipage}{6.8cm}
 \includegraphics[width=\textwidth]{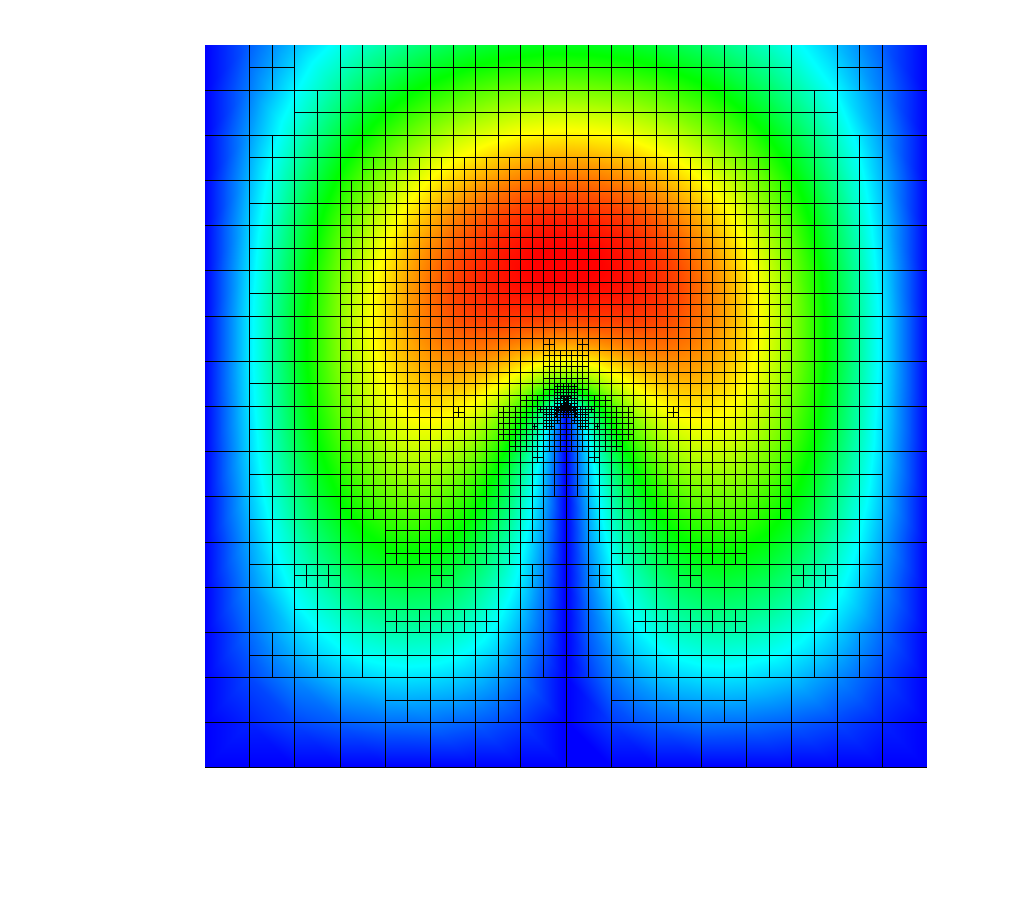}
 \end{minipage}
  \caption{Left: Comparison of Newton steps and computational time for the mesh-adaptive and the fully adaptive algorithms
on the test problem on the slit domain. Right: Optimal solution $u$ on an adaptively refined mesh for the slit domain.
\label{fig.slit}} 
\end{figure}

From the third mesh level, the mesh-adaptive algorithm needs two Newton steps on each mesh level, while the fully adaptive
one needs only one. As mentioned in Section~\ref{sec.algo} the latter strategy requires the solution of a dual problem after each Newton step, 
while in the mesh-adaptive algorithm 
a dual problem has to be solved only once after the last Newton step on each mesh level. As the system matrix for the primal and dual problem
have the same structure, the cost to solve them is comparable. Hence, two KKT systems have to be solved in the fully adaptive algorithm on
each of the finer mesh levels, compared to three for the mesh adaptive algorithm. This ratio of 2:3 can be observed in the computational times.
To reduce the error in ${\cal I}(q)$ below $10^{-4}$, the fully adaptive algorithm needs for example 131 seconds in contrast to 193 s
for the mesh adaptive one.
Using global refinement, 1157 s are needed to reach this tolerance, see Figure~\ref{tab.slit_global}.

Finally, we want to test the iteration error indicator $\eta_{KKT}$. As the iteration error of the Newton method 
decreases very quickly in the first Newton steps, this is barely visible in the previous calculations.
To investigate the iteration error indicator in more detail, we use a damped Newton iteration
with a damping factor that reduces the convergence rate of the iteration significantly.

Given an iterate $x^k$, the next iterate is defined by
\[
 x^{k+1} = x^k + \alpha_N \delta x,
\]
where $\delta x$ is the Newton direction and the damping parameter is chosen as $\alpha_N=0.5$. To keep the discretization
error small, we choose a very fine mesh with $526.850$ nodes. The results are shown in Table~\ref{KKTerror}. 
We observe that the discretization error $\eta_h$ varies slightly in the first iterates by less than 25\%
and stays then nearly constant. This indicates that the
two error estimators are asymptotically independent.
Moreover, we see that
the iteration error dominates the overall error until the ninth iteration. Up to then, the
effectivity index lies between $0.87$ and $1.16$. This indicates that the iteration error
indicator is very reliable in this test problem. After about 14 iterations, the discretization error
becomes dominant and the efficiency indices are about $0.32$ as observed in the previous test. 
At a certain iteration the two errors are comparable but with opposite sign, i.e.\ $\eta_{KKT} \approx -\eta_h$. 
A cancellation problem occurs and the efficiency index in step 11 becomes negative. This is typical when trying 
to split the error in different contributions and cannot be avoided.

\begin{table}
\centering
\small
\begin{tabular}{c|ccccccc} 
 Iteration & $\mathcal{I}(q)-\mathcal{I}(q_h)$  & $\eta_{KKT}$ & $\eta_{h}$ & $\mathcal{I}_{eff}$ \\
 \hline
1 & -2.09e-01 & -1.82e-01 & 1.81e-04 & 0.87  \\
3 & -6.69e-02 & -6.38e-02 & 2.15e-04 & 0.95  \\
5 & -1.81e-02 & -1.85e-02 & 2.30e-04 & 1.01  \\
7 & -4.14e-03 & -4.86e-03 & 2.35e-04 & 1.12  \\
9 & -4.90e-04 & -1.23e-03 & 2.36e-04 & 2.03  \\
11&  4.33e-04 & -3.08e-04 & 2.37e-04 & -0.16 \\
13&  6.64e-04 & -7.72e-05 & 2.37e-04 & 0.24  \\
15&  7.22e-04 & -1.93e-05 & 2.37e-04 & 0.30  \\
17&  7.37e-04 & -4.82e-06 & 2.37e-04 & 0.31  \\
19&  7.41e-04 &  2.46e-15 & 2.37e-04 & 0.32  \\
\end{tabular}
\caption{Splitting of the error indicators $\eta_h$ and $\eta_{KKT}$ for a damped Newton method
for the test problem on the slit domain.\label{KKTerror}}
\end{table}

As last example in this section, we study a problem in which the nonlinearity causes more Newton
steps to test the mesh-adaptive algorithm. Therefore, we introduce a further non-linearity in the state equation
\begin{alignat}{2}
\label{eq.u2}
\nonumber -\sigma \Delta u + u^2 &= f && \quad \text{ in } \Omega,\\
 \partial_n u &= q^2 \pi \sin(\pi x) && \quad \text{ on } \Gamma^{top},\\
\nonumber u&= 0 && \quad \text{ on } \partial\Omega \setminus \Gamma^{top}.
 \end{alignat}
We set the right-hand side to $f(x,y)=2 \pi^2 \sin(\pi x)\sin(\pi y)$ and use
the same objective functional $J(u,q)$ and the quantity of interest $\mathcal{I}(q)$ as above. 
We compare the number of Newton steps and the computational times of the fully adaptive and
the mesh-adaptive algorithm in Table~\ref{tab.u2}. On the finer mesh levels, the mesh-adaptive
algorithm requires 5 Newton steps per mesh level, which means that 6 KKT systems have to be solved,
while the fully adaptive algorithm refines the mesh after 2 Newton steps, i.e.\ 4 KKT systems have to be solved.
Therefore, we observe again a ratio of roughly 2:3 in the computational times.

\begin{table}
\centering
\begin{tabular}{cc|cc|cc} 
 & & \multicolumn{2}{c|}{mesh adaptive} &\multicolumn{2}{c}{fully adaptive} \\
\hline
\#dofs &$\mathcal{I}(q)-\mathcal{I}(q_h)$ & \#steps & time[s] & \#steps & time[s]\\
\hline
54    & 1.3e-01 & 6 & 2  & 3 & 1 \\
102   & 6.1e-02 & 4 & 5  & 1 & 3 \\
176   & 3.4e-02 & 5 & 9  & 2 & 6 \\
358   & 1.8e-02 & 5 & 13 & 2 & 9\\
578   & 9.3e-03 & 5 & 18 & 2 & 13\\
742   & 5.9e-03 & 5 & 28 & 2 & 20\\
1880  & 2.6e-03 & 5 & 39 & 2 & 27\\
2232  & 1.7e-03 & 5 & 65 & 2 & 44\\
6498  & 7.0e-04 & 5 & 98 & 2 & 65\\
8186  & 4.4e-04 & 5 & 138 & 2 & 91\\
10024 & 3.2e-04 & 5 & 221 & 2 & 161\\
\end{tabular}
 \caption{\label{tab.u2} Comparison of Newton steps and computational time for the mesh-adaptive and the fully adaptive algorithms
for the modified, non-linear state equation (\ref{eq.u2}).
}
\end{table}

\subsection{Optimal electrode design}
\label{sec.num_pipette}

In this section we consider the micro-pipette geometry described in Section~\ref{sec.application}. We use the Neumann boundary condition
\begin{align*}
 g(q,y) = \begin{cases}
         J_0(q), \quad &y=y_{\rm tip},\\
         J_k(q), \quad &| y - m_k| <s_k/2 \text{ for } k \in [1,...,s],\\
         0 \quad &\text{else}.
        \end{cases}
\end{align*}
with $J_k(q)$ as described in the appendix \eqref{Jk} and \eqref{Ik}. The optimization problem is given in (\ref{ConcreteOpt}), 
the continuous and discrete KKT system in (\ref{KKTsystem_continuous}) and (\ref{KKTsystem}). First, we study the 
configuration with two openings on each side and keeping the sizes $s_1$ and $s_2$ fixed. The design parameters are thus the vertical
position of the holes, in terms of their midpoints $m_1$ and $m_2$.

\noindent The objective functional is given by
\[
J(u,q)=\frac{1}{2} \int_{\Omega^s} (u-\hat{u})^2 dx + \frac{\alpha}{2} q^2  
\]
where $\hat{u}=5$, $\alpha=10^{-8}$ and a sub-domain $\Omega^s$ as shown in Figure~\ref{simdom}.
The domain $\Omega$ has a size of $40\mu m\times60\mu m$, the sub-domain
$\Omega^s$ of $32\mu m\times 35\mu m$ and the sizes of the 
openings are fixed as $s_0=1.5\mu m, s_1=1\mu m, s_2=2\mu m$. Moreover, the thickness of the micro-pipette wall is $d=0.5\mu m$, the inclination angle of the micro-pipette is
$\theta=22^{\circ}$, the
conductivity $\sigma \approx1.72(\Omega m)^{-1}$
and the applied current on the top of the micro-pipette is $\overline{I}=50\mu A$.
We approximate the function $g$ by the differentiable function $\tilde{g}$ defined in (\ref{expg}) with $\beta=2$. 
As goal functional, we consider again the error in the design parameter $\mathcal{I}(u,q)=q^2$.

In Figure~\ref{fig.pip_fullyadapt} we show the error estimators $\eta, \eta_h$ and $\eta_{KKT}$ as well as the effectivity
index $I_{eff}$ in the course of the fully adaptive algorithm on the left side. 
On each mesh level one Newton step was enough to reduce the 
iteration error $\eta_{KKT}$ below the discretization error $\eta_h$. On the other hand, $\eta_{KKT}$ is around $10^{-3}$ on the coarse mesh levels
and the corresponding Newton residual $\rho_{Newton}$ is far away from being below the tolerance $TOL_{KKT}$. Therefore, the purely mesh-adaptive 
algorithm needs more Newton steps (2-4) on each mesh level, see Table~\ref{tab.pip_comptime}.
Note that although the contribution of the iteration error to the goal functional is very small from the fourth mesh level ($\eta_{KKT}<10^{-9}$)
the Newton residual $\rho^k$ might still be much larger, such that the purely mesh adaptive algorithm
needs at least a second Newton step before $\rho_{Newton}<TOL_{KKT}$.

The effectivity indices are close to 1 on all fine mesh levels.
The error is well estimated on all mesh levels besides the third one with 3670 nodes. Here, the error $\mathcal{I}(q)-\mathcal{I}(q_h)$ changes its sign, 
which is not yet captured by the estimator. On the right, we show the optimal solution on an adaptively refined mesh. Note that most
of the refinement takes place around the two upper openings of the micro-pipette. The optimal positions of the openings found by the adaptive algorithm
for initial values $m_1^0=10\mu m$ and $m_2^0=20\mu m$ are
$m_1=6.1\mu m$ and $m_2=15.8\mu m$ above the tip of the micro-pipette, see also Figure~\ref{pip}.

\begin{figure}[t]
\centering
\small
\begin{minipage}{9cm}
\begin{tabular}{c|ccccc} 
\#dofs &$\mathcal{I}(q)-\mathcal{I}(q_h)$  & $\eta_h$ & $\eta_{KKT}$ & $\mathcal{I}_{eff}$ & $\rho_{Newton}$\\
\hline
2754  & -1.0e+01 & -4.9e+00 & -7.8e-03 & 0.46 & 9.1e-01 \\
3222  & -6.2e+00 & -5.6e+00 & -9.0e-03 & 0.91 & 1.3e+00 \\
3670  &  3.8e-01 & -2.2e-01 &  6.8e-04 & -0.58& 3.8e-02 \\
6426  &  2.7e-01 &  2.9e-01 &  2.0e-10 & 1.06 & 4.6e-05 \\
12506 &  1.1e-01 &  8.0e-02 &  3.7e-10 & 0.70 & 2.9e-05 \\
26506 &  6.1e-02 &  2.9e-02 &  4.3e-12 & 0.47 & 3.8e-06 \\
57068 &  2.7e-02 &  2.0e-02 & -3.2e-14 & 0.74 & 1.0e-06 \\
95894 &  1.3e-02 &  1.1e-02 &  5.9e-13 & 0.88 & 1.9e-07 \\
\end{tabular}
\end{minipage}
\begin{minipage}{6cm}
   \includegraphics[width=\textwidth]{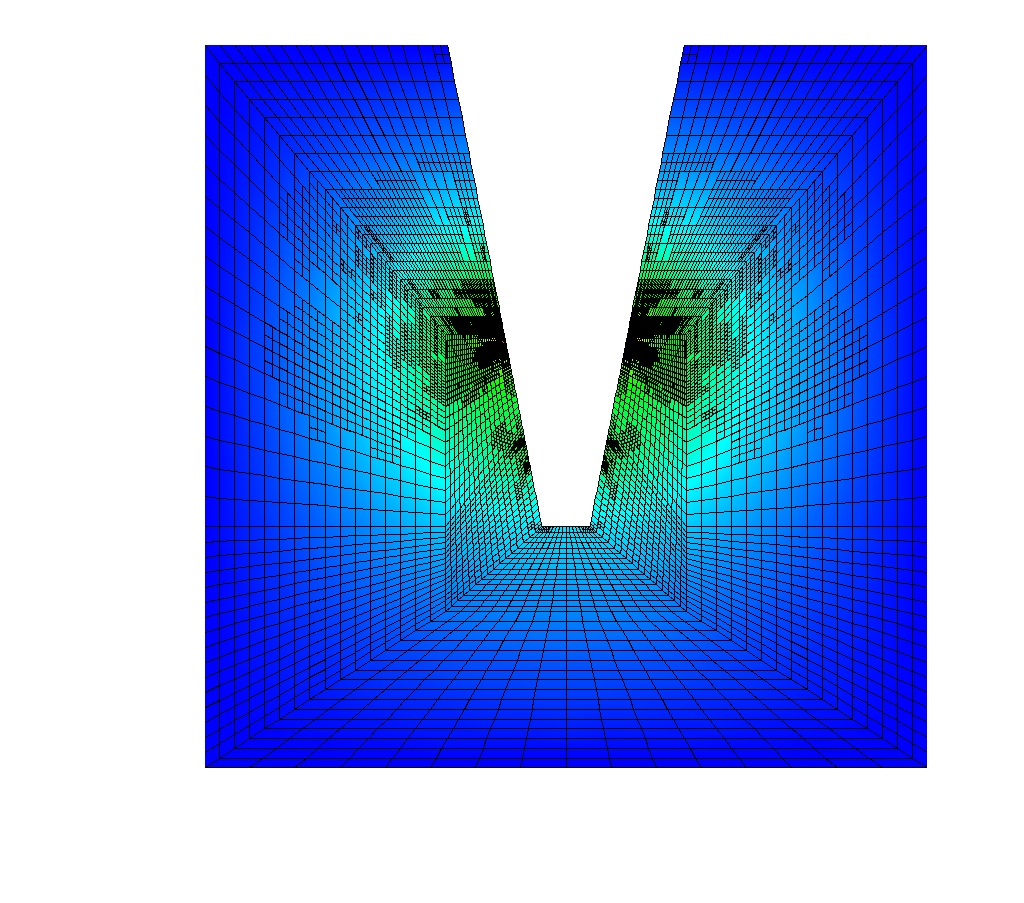}
\end{minipage}
   \caption{\textit{Left:} Error, error estimators and effectivity index for the fully adaptive algorithm applied to
the electrode problem\label{fig.pip_fullyadapt}.
\textit{Right:} Optimal state $u_{\rm opt}$ and finest adaptive mesh.}
\end{figure}

In Table~\ref{tab.pip_comptime}, we compare the mesh adaptive against the fully adaptive algorithm in terms 
of Newton steps and computational times. As mentioned before, the iteration error $\eta_{KKT}$ is reduced below the discretization error
$\eta_h$ already after the first Newton step in the fully adaptive algorithm, while 2 to 4 Newton steps are necessary
in the mesh-adaptive algorithm to reduce
the Newton residual below the tolerance $TOL_{KKT}=10^{-10}$. On the finer meshes we have to solve 2 primal and 1 dual KKT systems for 
the mesh adaptive algorithm and 1 primal and 1 dual system for the fully adaptive algorithm. Thus, the computational times show
again a ratio of roughly 2:3 on the finer meshes.

\begin{table}[t]
\renewcommand{\arraystretch}{1.2}
\centering
\begin{tabular}{cc|cc|cc} 
 & &\multicolumn{2}{c|}{mesh adaptive} &\multicolumn{2}{c}{fully adaptive}\\
 \hline
\#dofs &$\mathcal{I}(q)-\mathcal{I}(q_h)$  & \#steps & time[s] & \#steps & time[s]\\
\hline
2754  & -1.0e+01 &4 &34   &1&17 \\
3222  & -6.2e+00 &4 &80   &1&41 \\
3670  &  3.8e-01 &3 &128  &1&70 \\
6426  &  2.7e-01 &3 &205  &1&119 \\
12506 &  1.1e-01 &3 &335  &1&197 \\
26506 &  6.1e-02 &2 &520  &1&338 \\
57068 &  2.7e-02 &2 &831  &1&575 \\
95894 &  1.3e-02 &2 &1313 &1&943 \\
\end{tabular}
\caption{Number of Newton steps and computational times for the mesh-adaptive and the fully adaptive algorithm
for optimal electrode design.\label{tab.pip_comptime}}
\end{table}

On the left side of Figure~\ref{fig.pip_plots}, we compare the global refinement algorithm against the adaptive ones by plotting the
error against degrees of freedom. As the plots of the two adaptive algorithms are indistinguishable, we plot again only the errors 
for the fully adaptive algorithm. The global refinement strategy converges with
a rate slightly smaller than ${\cal O}(N^{-1})= {\cal O}(h^2)$,
while the adaptive algorithms converge significantly faster.
On the right side, we plot the error $|\mathcal{I}(q)-\mathcal{I}(q_h)|$ over the computational time for the two adaptive algorithms.
Due to the observations made above for the number of KKT systems to be solved, the two adaptive algorithms
show again a very similar asymptotic behavior in terms of computational times, with an advantage of roughly $33\%$ for the fully adaptive algorithm.

\begin{figure}[t]
\centering
 \includegraphics[width=0.45\textwidth]{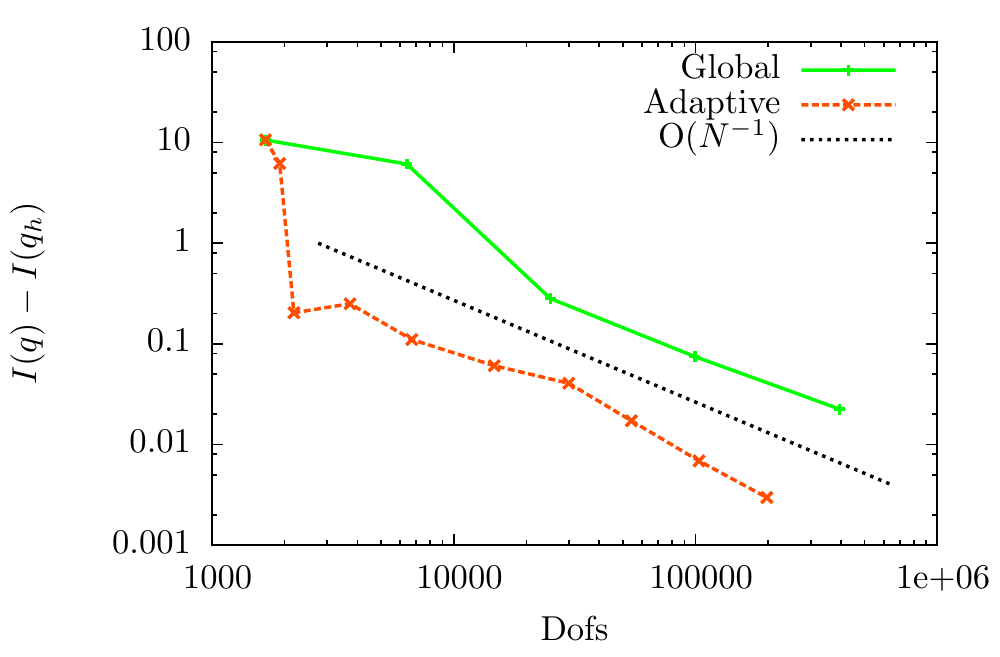}
 \includegraphics[width=0.45\textwidth]{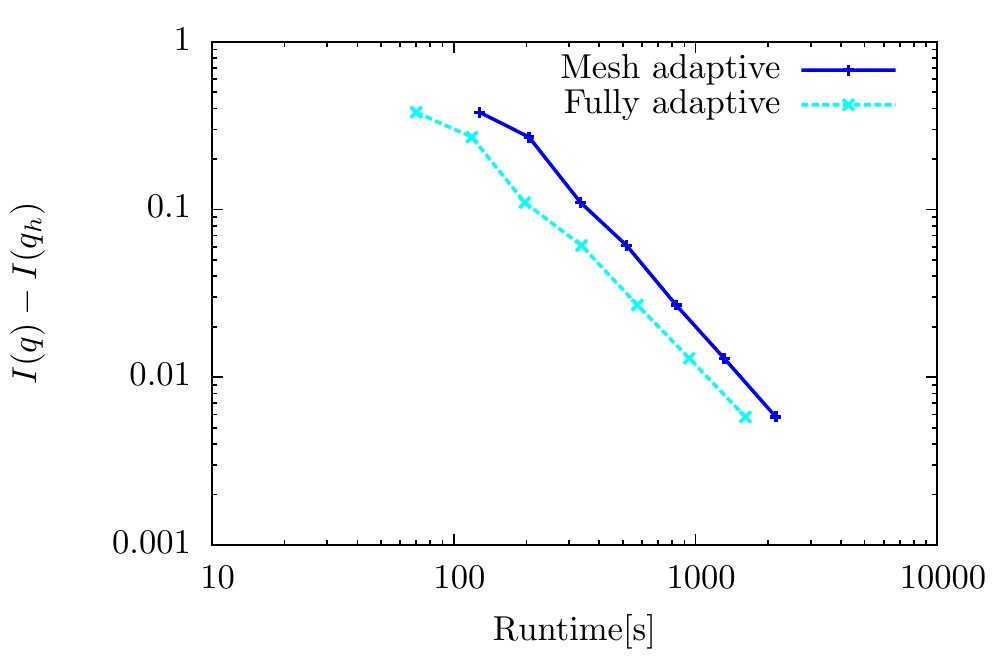}
 \caption{Left: Error plotted over degrees of freedom, for comparison we plot the reference rate ${\cal O}(N^{-1})$, where N denotes the
 number of degrees of freedom. Right: Error over computational time for the electrode problem
 with the global refinement, mesh adaptive and fully adaptive algorithm.\label{fig.pip_plots}}
 \end{figure}
  
Finally, we want to compare the effect of a different number of openings. In the case of only one opening at the bottom, nothing is to be optimized.
Solving the state equation to obtain the electric field yields an objective value of $J(u)\approx 11360$. The voltage distribution is shown in Figure~\ref{fig.pip_comp} 
on the left. The black contour line corresponds to the threshold $\overline{u}=4$.

For a fair comparison of micro-pipettes with one and two sets of openings, we have to optimize not only the position $m_k$, but also the size $s_k$ of the openings. Due to 
the complicated structure of the boundary fluxes $g_k(q)$ (see the appendix), we decided not to implement the additional derivatives that would be necessary
for a simultaneous optimization of size and position.
Instead, we alternately optimize size and position by keeping the respective other parameters fixed, see Table~\ref{tab.SimultaneousOpt}
for the case of two openings.

The optimal parameters the optimization algorithm has found were sizes of $s_1=0.35\mu m$, $s_2=0.3\mu m$ and positions of $m_1=5.9\mu m$ and $m_2=19.7\mu m$
in the case of two openings and $s_1=0.31\mu m$ and $m_1=18.2\mu m$ for one opening.
The optimal functional value is given by $J(u,q)\approx 6217$ for one 
set of holes and $J(u,q)\approx 6018$ for two sets. While we obtain a reduction of more than $45\%$ between the case without a hole and the case with one set of holes, the reduction
between one and two sets of holes is only around $3.2\%$. This can also be seen from the voltage distribution in Figure~\ref{fig.pip_comp}.

These results show that the modified micro-pipettes yield a significantly larger region where cell membranes are made permeable. Moreover, the results indicate that more than one set
of holes does not bring a significant advantage anymore, as a relatively large region around the micro-pipette is activated already by one hole per side. The second
hole that is placed relatively close
to the tip of the micro-pipette by the optimization algorithm, seems to have a much smaller effect. 
However, the overall voltage distribution is more uniform and the peak potential regions at 
the holes (red) are reduced, which may have an advantageous effect for the health of cells 
(without losing any/much of the overall electroporation volume).

\begin{table}
\renewcommand{\arraystretch}{1.2}
\centering
\begin{tabular}{c|cc|cc|c} 
Step & $m_1$ & $m_2$ & $s_1$ & $s_2$ & $J(u,q,s)$\\
\hline
0 & 10.0 & 20.0 & 0.41 & 0.30 & 6214 \\
1a & 4.8 & 19.7 & " & " & 6021 \\
1b & " & " & 0.39 & 0.30 & 6020 \\
2a & 5.1 & 19.7 & " & " & 6020 \\
2b & " & " & 0.38 & 0.30 & 6019 \\
\vdots &\vdots&\vdots&\vdots&\vdots&\vdots\\
\hline
OPT & 5.9 & 19.7 & 0.35 & 0.30 & 6018 \\
\end{tabular}
\caption{Simultaneous optimization of size and position of openings\label{tab.SimultaneousOpt}.
Alternately, the position of the holes are optimized in step 'a' and the size of the holes in step 'b',
while keeping the other respective parameters fixed.}

\end{table}

\begin{figure}[t]
\centering
\subfigure{
 \includegraphics[width=4cm]{opt_0.png}
}
\subfigure{
 \includegraphics[width=4cm]{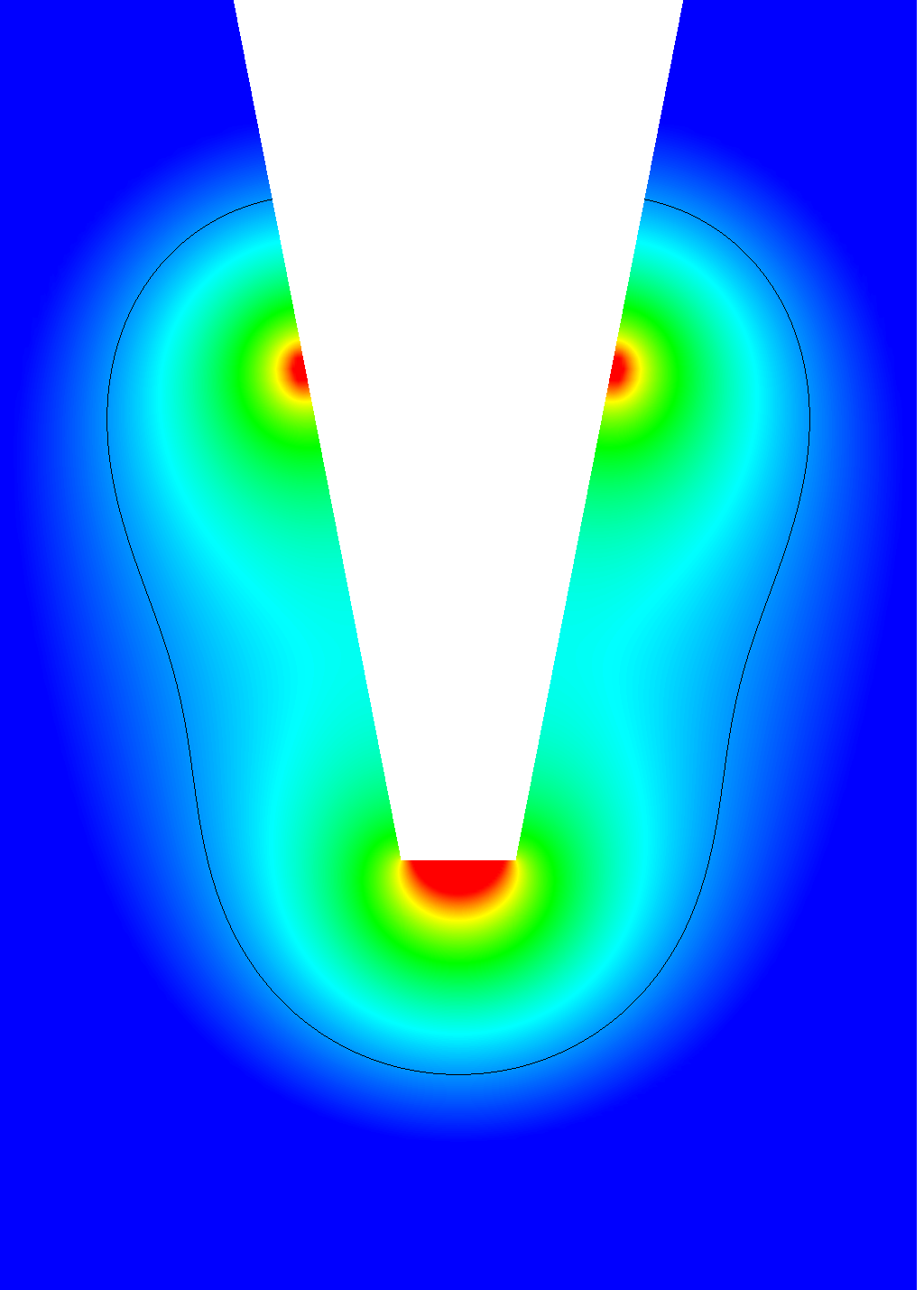}
}
\subfigure{
 \includegraphics[width=4cm]{opt_2.png}
 }
 \caption{Optimal results for the electrode problem with respect to radius and sizes of the openings for 0, 1 and 2 sets of openings.
 The coloring represents the voltage, the black contour line is the threshold voltage $u=\overline{u}$.\label{fig.pip_comp}}

\end{figure}

\section{Conclusion \& Outlook}
\label{sec.conclusion}

We have presented an adaptive optimization algorithm for the optimal design of a micro-pipette for electroporation used for neuronal networks tracing.
The main contribution is the derivation of the goal-oriented strategy that allows to steer the number of Newton steps balancing the discretization error and the solution error with respect to a (nonlinear) functional that represents a quantity of interest for the specific optimization problem. 

Possible extension of this approach is the balance of the linearization error and the error due to the linear solver as in \cite{RannacherVihharev:2013}. Furthermore, more sophisticated optimization algorithms as the one presented in \cite{Ulbrich:2011} can be considered in this framework to increase the robustness of the optimization method.
In addition, globalization techniques as the one presented in \cite{Wihler:2015} that allow to control the convergence behavior of the Newton method become essential in certain practical cases and should be considered in future works.

Another possible extension is to include control and/or state constraints in the formulation. This is possible, as already mentioned, without significant changes in the approach presented here.
Furthermore, the algorithm presented can be used and adapted to many other applications for which a Newton-type method can be applied.

Using the adaptive algorithm we have shown that quantitative improvement of the micro-pipette design can be obtained by a model-based optimization method. 
This approach is a promising tool with a significant impact in the neurosciences.

\section{Appendix}

\begin{figure}[t]
 \centering
 \input{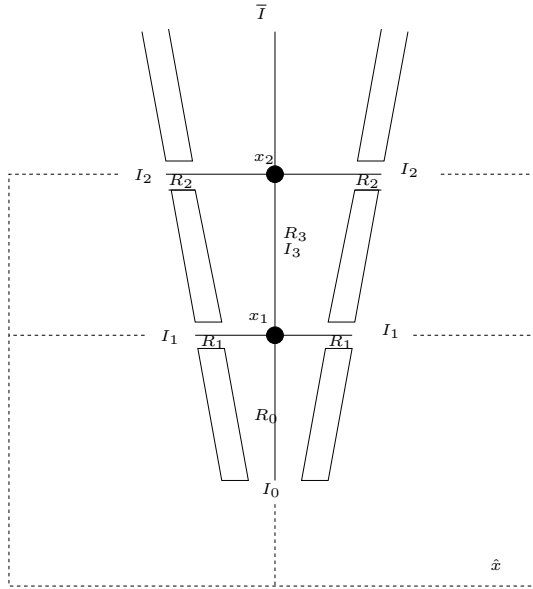}
 \caption{ \label{stromkreis} Scheme of the electric circuit around the micro-pipette}
\end{figure}

In the appendix, we derive the flux function $g$ and its dependency on sizes $s_k$ and positions $m_k$ of
the $k$-th hole, $k=0...s/2$. A scheme of the electric circuit is shown in Figure~\ref{stromkreis}. 
For simplicity, we present only the case of a micro-pipette with two holes on each side.
The derivation of corresponding formulas for a different number of holes is analogous.

We assume that a fixed current $\overline{I}$ is applied at
the top of the micro-pipette. We calculate the current $I_k$ that flows out of the micro-pipette at the holes $k=0,\dots,2$.
The flux function $g_k$ at hole $k$ is then given by the current density $J_k$
\begin{align*}
 g_k= J_k = \frac{I_k}{|\Gamma_k|} \quad \text{ on } \Gamma_k.
\end{align*}
The micro-pipette is filled with a conducting liquid. We assign a specific resistance $R_j$, $j=0, \dots, 3$, to each of the parts 
of the micro-pipette. The resistances of the conducting liquid in the small holes on the left and right in between the isolating wall are 
denoted by $R_1$ and $R_2$, the resistances of the parts in the interior of the micro-pipette by 
$R_3$ and $R_0$. Denoting the thickness of the wall by $d$, the resistance of a hole 
is given by
\begin{align*}
 R_k = \rho \frac{d}{\pi s_k^2} \quad (k=1,2),
\end{align*}
where $\rho= 1/\sigma$ is the electrical resistivity. To calculate the resistance of 
the conical part below $x_2$, we introduce the notation $a(x)$ for the area inside the micro-pipette
at position $x$, see Figure~\ref{cone}. The area $a_0=a(0)$ of $\Gamma_0$ at the tip of the micro-pipette is given by $a(0)=\pi s_0^2$, the area
$a_1=a(m_1)$ of $\Gamma_1$ at the first hole by 
\begin{align*}
 a(m_1)=\pi (s_0 + \tan(\theta)m_1)^2, 
\end{align*}
where $\theta$ is the inclination angle of the micro-pipette.
The resistance of the conical part below $x_2$ is then given by (see e.g.~\cite{Griffiths:2005})
\begin{align*}
 R_0 = \rho \int_0^{m_1} \frac{1}{a(x)^2} dx = \frac{\rho}{\pi} \rm{cot} (\theta) \left( \frac{1}{s_0}  -\frac{1}{s_0 + m_1 \tan(\theta)}\right).
\end{align*}

\begin{figure}
 \centering
 \input{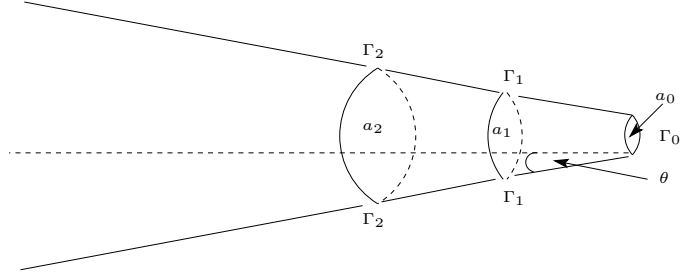}
 \caption{ \label{cone} Scheme of the tip region of the micro-pipette including area elements to calculate the resistances.}
\end{figure}

\noindent Similarly, we get for the resistance $R_3$ of the part between the points $x_1$ and $x_2$
\begin{align*}
 R_3 = \frac{\rho}{\pi} \rm{cot} (\theta) \left(\left(s_0 + m_1 {\rm tan}(\theta)\right)^{-1} - \left(s_0 + m_2 {\rm tan}(\theta)\right)^{-1}\right). 
\end{align*}

The voltage difference between point $x_2$ and a point $\hat{x}$ far off the micro-pipette can be used to derive
the following formula by using Ohm's law (see Figure~\ref{stromkreis})
\begin{align}
I_2 \cdot R_2 = I_3 \cdot R_{0,1,3} \quad  (= u(x_2) - u(\hat{x}) ) \label{URI1}
\end{align}
Here, $R_{0,1,3}$ stands for the total resistance of the parts $R_0$, $R_1$ and $R_3$ which is given by
\begin{align*}
 R_{0,1,3} = R_3 + (R_0^{-1} + 2 R_1^{-1})^{-1}.
\end{align*}
Furthermore, by Kirchhoff's current law the current $\overline{I}$ splits at point $x_2$ to 
\begin{align}
 \overline{I} = I_3 + 2 I_2. \label{split1}
\end{align}
(\ref{URI1}) and (\ref{split1}) can be solved for the two unknowns $I_2$ and $I_3$.
In the same way, it holds at point $x_1$
\begin{align}
 R_0 \cdot I_0 = R_1 \cdot I_1 \label{URI2}
\end{align}
and
\begin{align}
 I_3 = 2 I_1 + I_0. \label{split2}
\end{align}
Given $I_3$, (\ref{URI2}) and (\ref{split2}) define $I_0$ and $I_1$. Inserting the formulas for the resistances, a direct calculation results in
\begin{align}
 I_0 &= \overline{I} \frac{R_1 R_2}{\left(R_2 + 2 R_{0,1,3} \right) \left( R_1 + 2 R_0\right) } = 
        \overline{I} \frac{\left( s_0\,c + m_1 \right)^2 \left(s_0\,c + m_2 \right) d^2 \, s_0}{T(q,s)},\nonumber \\
 I_1 &= \overline{I} \frac{R_0 R_2}{\left(R_2 + 2 R_{0,1,3} \right) \left( R_1 + 2 R_0\right) } =
        \overline{I} \frac{\left( s_0\,c + m_1 \right) \left(s_0\,c + m_2 \right) d \, c \, m_1 \, s_1^2}{T(q,s)},\label{Ik}\\
 I_2 &= \overline{I} \frac{R_{0,1,3}}{R_2 + 2 R_{0,1,3}} 
     = \overline{I} \frac{\left( m_2 d s_0^2 c^2 + 2 c m_2 d s_0 m_1+ 2 m_2 s_1^2 c^2 m_1-2 s_1^2 c^2 m_1^2+d m_1^2 m_2 \right)\,c \, s_2^2}{T(q,s)} 
\nonumber
\end{align}
with $c:={cot} (\theta)$ and
\begin{align*}
T(q,s)&={{\it s_0}}^{4} c^{3}{d}^{2}+2\,{{\it s_0}}^{3} c^{2}{d}^{2}{\it m_1}+2\,d{{\it s_0}}
^{2} c^{3}{\it m_1}\,{{\it s_1
}}^{2}+{{\it s_0}}^{3} c ^{2}{
\it m_2}\,{d}^{2}+2\,{{\it s_0}}^{2}c {\it m_2}
\,{d}^{2}{\it m_1}\\
&+2\,d{\it s_0}\,c^{2}{\it m_2}\,{\it m_1}\,{{\it s_1}}^{2}+
{{\it m_1}}^{2}{{\it s_0
}}^{2}c {d}^{2}+2\,d{{\it m_1}}^{2}{\it s_0}\,
 c ^{2}{{\it s_1}}^{2}+{{\it 
m_1}}^{2}{\it m_2}\,{d}^{2}{\it s_0}\\
&+2\,d{{\it m_1}}^{2}{\it m_2}\, c{{\it s_1}}^{2}+2\, c^{3}{{\it s_2}}^{2}{\it m_2}\,d{{\it s_0}}^{2}+4\,
 c ^{2}{{\it s_2}}^{2}{\it m_2}
\,d{\it s_0}\,{\it m_1}+4\, c ^
{3}{{\it s_2}}^{2}{\it m_2}\,{\it m_1}\,{{\it s_1}}^{2}\\
&-4\, c^{3}{{\it s_2}}^{2}{{\it m_1}}^{2}{{\it 
s_1}}^{2}+2\,c {{\it m_1}}^{2}d{{\it s_2}}^{2}{
\it m_2} .\\ 
\end{align*}

As we use a two dimensional setting the size of $\Gamma_k$ is given by $|\Gamma_k|=s_k$.
The flux function $g_k$ at hole $k$ is thus given by
\begin{align}
 g_k= J_k = \frac{I_k}{s_k} \quad \text{ on } \Gamma_k. \label{Jk}
\end{align}

\section*{Acknowledgements}
\noindent We are grateful to Prof. Andreas T. Schaefer for generous collaborative support and fruitful discussions of this work.\\
T.C.\ was supported by the Deutsche Forschungsgemeinschaft (DFG) through the project  CA 633/2-1.

\bibliographystyle{abbrv}
\bibliography{arbeit}

\end{document}